\documentclass[reqno]{amsart}
\usepackage{amsmath}
\usepackage{amsfonts}
\usepackage{amstext}
\usepackage{amsbsy}
\usepackage{amsopn}
\usepackage{amsxtra}
\usepackage{upref}
\usepackage{amsthm}
\usepackage{amsmath}
\usepackage{amssymb}
\usepackage{amscd}
\usepackage{euscript}
\usepackage{graphicx}
\usepackage[cspex,bbgreekl]{mathbbol}
\usepackage{enumerate}
\usepackage[bookmarks=false]{hyperref}
\usepackage{mathrsfs}
\usepackage{enumitem}

\newtheorem{prop}{Proposition}[section]
\newtheorem{rem}{Remark}[section]
\newtheorem{lema}{Lemma}[section]
\newtheorem{defi}{Definition}[section]
\newtheorem{teo}{Theorem}[section]
\newtheorem{maintheorem}{Theorem}
\newtheorem{eje}{Example}[section]
\newtheorem{coro}{Corollary}[section]

\newtheorem*{claim*}{Claim}

\def\eps{\varepsilon}
\def\phi{\varphi}
\def\R{{\mathbb R}}

\def\N{{\mathbb N}}

\def\P{{\mathcal P}}

\def\F{{\mathcal F}}
\def\D{{\mathcal D}}
\def\M{{\mathcal M}}

\def\T{{\mathcal T}}

\def\es{{\emptyset}}
\def\sm{\setminus}

\def\crit{{\mathcal Cr}}

\def\bd{\partial }
\def\le{\leqslant}
\def\ge{\geqslant}
\def\st{such that }

\def\F{\mathcal{F}}
\def\M{\mathcal{M}}

\def\level{\text{lev}}
\def\cyl{{\rm C}}

\title[Thermodynamic formalism for interval maps: inducing schemes]{Thermodynamic formalism for interval maps: inducing schemes}
\date{\today}

\begin{thanks}
{  G.I. was partially  supported by  the Center of Dynamical Systems and Related Fields c\'odigo ACT1103 and by Proyecto Fondecyt 111004.}
\end{thanks}


\author{Godofredo Iommi}
\address{Facultad de Matem\'aticas,
Pontificia Universidad Cat\'olica de Chile (PUC), Avenida Vicu\~na Mackenna 4860, Santiago, Chile}
\email{\href{mailto:giommi@mat.puc.cl}{giommi@mat.puc.cl}}
\urladdr{\url{http://www.mat.puc.cl/~giommi/}}
\author{Mike Todd} \address{
Mathematical Institute,
University of St Andrews,
North Haugh,
St Andrews,
KY16 9SS,
Scotland}
\email{\href{mailto:mjt20@st-andrews.ac.uk}{mjt20@st-andrews.ac.uk}}
\urladdr{\url{http://www.mcs.st-and.ac.uk/~miket/}}

\begin{document}

\begin{abstract}
This survey article concerns inducing schemes in the context of interval maps. We explain how the study of these induced systems allows for the fine description of, not only, the thermodynamic formalism for certain multimodal maps, but also of its multifractal structure. %In particular, we review results from \cite{IomTod10,IomTod11} where the inducing schemes  technique was heavily used.
\end{abstract}

\maketitle

\section{Introduction}
\label{sec:intro}

A major breakthrough in theory of dynamical systems was the realisation that complicated behaviour exhibited by certain systems can be studied using probabilistic methods. This approach is based on what is usually called \emph{ergodic theory} and at its core is the study of existence and properties of dynamically relevant measures, the so-called \emph{invariant measures}. It was Poincar\'e \cite{poi} who realised that the simple existence of a finite invariant  measure yields non-trivial information on the orbit structure. Invariant measures exist under very weak assumptions on the phase space and on the system \cite[Corollary 6.9.1]{Waltbook}, hence the power of the approach.

There is however a difficulty, in that a large class of interesting dynamical systems have many invariant measures. Indeed, in simple settings such as the full-shift on two symbols  the set  of invariant probability measures is a Poulsen simplex  \cite{gw, los}, that is, an infinite dimensional, convex and compact set for which the extreme points are dense on the whole set. It is, therefore, an important problem to find criteria to choose \emph{relevant} invariant measures. This is where the \emph{thermodynamic formalism} comes into play. Indeed, given  a dynamical system $T:X \to X$ and a continuous function $\phi : X\to \mathbb{R}$ (the \emph{potential}) the \emph{topological pressure} can be defined in the following way
\begin{equation} \label{def:pres}
 P(\phi) = \sup \left\{ h(\mu) + \int \phi \ d \mu : \mu \in \mathcal{M}_T \textrm{ and } -  \int \phi \ d \mu< \infty \right  \},
 \end{equation}
where $h(\mu)$ denotes the entropy of the measure $\mu$ (see \cite[Chapter 4]{Waltbook}) and $\mathcal{M}_T$ denotes the space of $T-$invariant probability measures. A measure $\nu \in \mathcal{M}_T$ attaining the above supremum is called an \emph{equilibrium state}. Proving existence and uniqueness of equilibrium states is one of the major problems in the theory of thermodynamic formalism. This approach to choosing relevant measures as equilibrium states is a set of ideas and techniques which derive from statistical mechanics and that was brought into dynamics in the early seventies by Ruelle and Sinai among others  \cite{Dobr, Sinai, bo, Kel98, Ru_book,Waltbook}.

Properties and regularity of thermodynamic formalism depend, essentially, on two factors. On the one hand the regularity of the potential considered; and on the other the hyperbolicity/expansiveness  of the underlying dynamical system. If the potential $\phi$ is regular enough (say H\"older continuous) and the system is sufficiently hyperbolic (say uniformly expanding) then there exists a unique equilibrium state for $\phi$ and it has strong ergodic properties. Moreover, the pressure function $t \mapsto P(t\phi)$ is real analytic on $\R$ (see \cite{Ru_book}).

The main focus of this survey article will be to study thermodynamic formalism for  interesting systems that are far from being uniformly hyperbolic and for natural potentials that are not even continuous. That  is, the two main features which govern the thermodynamic formalism behave in a poor way. However, we will show that even in that setting we are able to describe the pressure function in great the detail. We will survey results from  the papers \cite{IomTod10} and \cite{IomTod11}.  The main technique we use is that of \emph{inducing}. The idea is to associate to our dynamical system a generalisation of the  first return map (the \emph{induced system}) that not only is uniformly expanding but also is a full-shift on a countable alphabet. A price one has to pay in order to obtain a Markov structure and expansiveness is that our dynamical system is no longer defined over a compact set. This is serious matter since in that context the thermodynamic formalism, even for regular potentials, can be irregular (see Section \ref{sec:CMS}). Another issue is that it is possible that the inducing procedure `throws away' parts of the system which are important in obtaining thermodynamic quantities.
We will discuss the advantages and disadvantages of this approach, explain the main steps in implementing this, and state the results proved in the above papers.

 The main class of dynamical systems we consider here, denoted by $\mathcal F$, is the collection of $C^2$ multimodal interval maps $f:I \to I$, where $I=[0,1]$, satisfying:

\begin{enumerate}[label=({\alph*}),  itemsep=0.0mm, topsep=0.0mm, leftmargin=7mm]
\item the critical set $\crit = \crit(f)$ consists of finitely many critical points $c$ with critical order $1 < \ell_c < \infty$, i.e., there exists a neighbourhood $U_c$ of $c$ and a $C^2$ diffeomorphism $g_c:U_c \to g_c(U_c)$ with $g_c(c) = 0$
     $f(x) = f(c) \pm |g_c(x)|^{\ell_c}$;
\item $f$ has negative Schwarzian derivative, i.e., $1/\sqrt{|Df|}$ is convex;
\item $f$ is topologically transitive on $I$ (i.e., there exists a dense orbit);
\item $f^n(\crit)  \cap f^m(\crit)=\es$ for $m \neq n$.
\end{enumerate}

Note that the class $\F$ includes transitive Collet-Eckmann maps, that is maps where
$|Df^n(f(c))|$ grows exponentially fast in $n$. In particular, the set of quadratic maps in $\F$ has positive Lebesgue measure in the parameter space of quadratic maps (see \cite{BenCar, Jak}). Our main application will be to maps in this family, although we will also remark on extensions (see Section \ref{ssec:cusps}). It will sometimes be useful to restrict our family further.  For an interval map $f:I\to I$, the \emph{metric attractor} is a set $A\subset I$ such that the corresponding set $B(A):=\{x\in I:\omega(x)\subset A\}$ has positive Lebesgue measure and there is no proper subset of $A$ with this property (here $\omega(x)$ is the set of accumulation points of the orbit of $x$).  On the other hand, $A$ is a \emph{topological attractor} if $B(A)$ is residual.  The map $f$ has a \emph{wild attractor} if the there is a metric attractor which is not a topological attractor\footnote{By part (c) of the definition of the class $\F$, the topological attractor of a map $f\in \F$ is always the whole of $I$.}.  We call the set of interval maps without wild attractors \emph{tame}, and denote those in $\F$ by $\F_T$.  We will also be interested in $f\in \F$ which admit a measure $\mu\in \M_f$ which is absolutely continuous w.r.t.\ Lebesgue, which we call an \emph{acip}.  It can be shown that if $f$ has an acip then $f\in \F_T$.

We will be particularly interested in the \emph{geometric potentials} $x \mapsto -t \log|Df(x)|$.   %One of the main reasons for this is that the thermodynamic formalism for maps in $\F$ is fundamentally determined by the growth properties of the system.  Roughly speaking, if there is a lot of exponential growth, then the system behaves as in the classical uniformly hyperbolic theory.  Moreover,
The equilibrium states corresponding to the geometric potentials capture important geometric features that allow us to study the fractal geometry of dynamically relevant subsets of the phase space (see Theorem \ref{thm:main lyap}).  These measures are supported on sets for which the expansion properties of the system are simple, in the sense that expansion rates are constant for almost every point.  This is of particular importance because  when there is strong expansion  the system behaves as in the classical uniformly hyperbolic theory. The geometric potentials are also related to the existence of an acip as we will see in  Section~\ref{subsec:conf}.
The presence of critical points, which leads to significant distortion problems with both the map and the geometric potentials, as well as the non-Markov structure make these systems particularly interesting, and particularly suited to study using inducing methods
%Actually, the existence of absolutely continuous invariant measures with respect to the Lebesgue measure (known as an \emph{acip}) is related to the existence of equilibrium states for $-\log|Df(x)|$.

Considering the above discussion, we define, the various kinds of Lyapunov exponents.  Note that we will often suppress the dependence of these quantities on the map $f$ in the notation.  Given a smooth interval map $f$, the \emph{lower (pointwise) Lyapunov exponent} and  \emph{upper (pointwise) Lyapunov exponent} at $x\in I$ are defined as
$$\underline\lambda(x)=\underline\lambda_f(x):=\liminf_{n\to\infty}\frac{1}{n} \log | Df^n (x)|  \text{ and }  \overline\lambda(x)=\overline\lambda_f(x):=\limsup_{n\to\infty}\frac{1}{n} \log |Df^n(x)|,$$
respectively.  When these values are equal, we denote their common value, the \emph{(pointwise) Lyapunov exponent} at $x$, by $\lambda_f(x)$.  Note that in some cases, we will deal with only piecewise smooth functions, in which case we disregard points where the derivative is not defined.  We define the global infimum and supremum of these quantities by
$$\lambda_{\inf}=\lambda_{\inf}(f):=\inf\{\lambda_f(x):x\in I \text{ and this limit exists}\},$$
and
$$ \lambda_{\sup}= \lambda_{\sup}(f):=\sup\{\lambda_f(x):x\in I \text{ and this limit exists}\}.$$
%We stress that, even if it does not appear explicitly, the above values depend on the map $f$ we are considering.
We also define the \emph{Lyapunov exponent} of $\mu\in \M_f$ to be
$$\lambda(\mu)=\lambda_f(\mu):=\int\log|Df|~d\mu.$$
Note that if $\mu$ is ergodic then $\lambda(x)=\lambda(\mu)$ for $\mu$-a.e. $x\in I$.  We also define
$$\lambda_m=\lambda_m(f):=\inf\{\lambda_f(\mu):\mu\in \M_f\} \text{ and } \lambda_M=\lambda_M(f):=\sup\{\lambda_f(\mu):\mu\in \M_f\}.$$
As in \cite[Lemma 4.1]{IomTod11}, $\lambda_M(f)=\lambda_{\sup}(f)$ for $f\in \F$.  However, although by analogue with the complex setting in \cite[Lemma 6]{GelPrzRam10}, we expect $\lambda_m(f)= \lambda_{\inf}(f)$,  we only know that $\lambda_m(f)\ge \lambda_{\inf}(f)$.  %\textbf{MT:I thought that Juan had shown these were equal, but it seems not.}
In this setting, \cite{Prz93} implies that $\lambda_m(f)\ge 0$ for all $f\in \F$.  Define the `expanding' and `good' elements of $\F$ by
%$$\F_E:=\{f\in\F:\lambda_m(f)>0\} \text{ and } \F_G:= \{f\in \F: |(f^n)'(f(c))|\to \infty\  \forall\ c\in \crit\}.$$
$$\F_E:=\{f\in\F:\lambda_m(f)>0\} \text{ and } \F_G:= \{f\in \F: \text{ there exists an acip} \}.$$
Some of the key behaviour here can be determined by the growth along the critical orbits.
In the unimodal case, where there is only one critical point $c$, the main result of \cite{NowSan98} implies that $f\in \F_E$ if and only if $\underline\lambda_f(c)>0$.  In \cite{BRSS}, it was shown that even in the multimodal case $|Df^n (f(c))|\to \infty$ for all $c\in \crit$ implies $f\in \F_G$.  Moreover, all unimodal maps with critical order 2 are tame, see \cite{Lyu94,GraSanSwi04}. However, when there is more than one critical point, the situation is obscured by the possibility of mixed behaviour between different critical points.

\subsection{Conformal measures}  \label{subsec:conf}

We now fix some notation and give an important definition. Denote $\phi_t:=-t\log|Df|$. The pressure function will be denoted by $p(t)=p_f(t):=P(-t\log|Df|)$. Given a potential $\phi:I \to\R$, we say that the measure (not necessarily invariant) $m$ is a \emph{$\phi$-conformal measure} if for any measurable set $A\subset I$ such that $f:A\to f(A)$ is a bijection, we have
$$m(f(A))=\int_A e^{-\phi}~d m.$$
We also write $dm(f(x))=e^{-\phi(x)}dm(x)$.
Before stating our first main Theorem, we observe that if $f\in \F$  then $p(1)=0$  and Lebesgue measure $m_1$ is $-\log|f'|$-conformal,  see \cite{Led81, NowSan98}.  Furthermore, by \cite{Led81}, if $\mu\in \M_f$ has $\lambda(\mu)>0$, then $\mu\ll m_1$ (i.e. $\mu$ is an acip) if and only if $\mu$ is an equilibrium state for $-\log|f'|$.    Theorem~\ref{thm:main eq} below concerns a more general set of measures.

\subsection{Main results}

Let $t^+=t^+(f):= \sup\{t : p_f(t) > -\lambda_m(f) t  \}$. The following result was proved in \cite{IomTod10}.

\begin{maintheorem}
Let $f\in \F$.  Then  $t^+\in (0, \infty]$ and:
\begin{enumerate}[label=({\alph*}),  itemsep=0.0mm, topsep=0.0mm, leftmargin=7mm]
\item for each $t <t^+$, there exists a unique equilibrium state $\mu_t$ for $\phi_t$;
\item for each $t <t^+$, there exists a unique $(\phi_t-p(t))$-conformal measure $m_t$.  Moreover, $\mu_t\ll m_t$;
\item $p(t)$ is strictly decreasing, strictly convex and $C^1$ on $(-\infty, t^+)$;
\end{enumerate}
Furthermore,
\begin{enumerate}[resume*]
\item if $f\in \F_{E}$ then $t^+>1$;
\item if $f\in \F_G\sm\F_{E}$, then $t^+=1$, $D^-p(1)<1$ and there exists an acip;
\item if $f\in \F_T$ then $t^+=1$. If moreover $f\in \F_T\sm \F_G$ then $p'(1)=0$;
\end{enumerate}
\label{thm:main eq}
\end{maintheorem}

\begin{rem}
To the best of our knowledge, the behaviour of the pressure function is not known for $f\in \F\sm\F_T$.  However, the works \cite{AviLyu08} and \cite{BruTod12} suggest that $t^+<1$, $p'(t^+)=0$ and there is no equilibrium state for $-t^+\log|Df|$.
\end{rem}

It is an interesting fact that despite the strong lack of expansiveness of the system (principally due to the presence of critical points), we are able to describe the pressure function $p(t)$ in  great detail. Note that for $t > t^+$ the pressure function $p(t)$ is linear. The point $t^+$ has been called \emph{freezing point} \cite{RivPrz11}.  The language here comes from statistical mechanics; for example, comparable transitions occur at a freezing point in Fischer-Felderhof models, see \cite{Fis72}. We stress that if $t^+< \infty$ then the pressure function exhibits a  \emph{phase transition} at $t=t^+$, that is, the pressure function is not real analytic at $t=t^+$.

Theorem \ref{thm:main eq} is a generalisation of the following results.   \cite{BruKel98} deals with the case of unimodal Collet-Eckmann maps  (recall that this means $\underline\lambda_f(f(c))>0$) for a small range of $t$ near $1$, one advantage of their results being analyticity of pressure in that range.  \cite{PesSen08} considers a subset of Collet-Eckmann maps, but for all $t$ in a neighbourhood of $[0, 1]$.  \cite[Theorem 1]{BruTod09} applies to a class of non-Collet Eckmann multimodal maps with $t$ in
a left-sided neighbourhood of $1$, again proving analyticity of pressure in that range. For rational maps, a result of this kind has been obtained by Przytycki and Rivera-Letelier \cite{RivPrz11}, in which they consider the full range of $t$ and obtain analyticity of pressure in that range.  A more detailed study of the behaviour of the pressure function at the phase transition $t^+$, inspired by \cite{MakSmi03}, was carried out in \cite{CorRiv12}.

As we have seen, the expanding properties of the systems coded by the Lyapunov exponents are of fundamental importance in order to describe the dynamical properties of the system. It is, therefore,  natural to try to characterise  the sets
\begin{equation*}
J(\lambda):= \Big{\{} x \in I :  \lambda(x) = \lambda \Big{\}}.
\end{equation*}
for $\lambda\in [\lambda_{\inf}, \lambda_{\sup}]$.  Defining
$$J':= \Big{\{} x \in I :  \lim_{n \to \infty} \frac{1}{n} \log | Df^n(x)| \text{ does not exist} \Big{\}},$$
we can write
$$I=J'\cup\bigcup_{\lambda\in  [\lambda_{\inf}, \lambda_{\sup}]}J(\lambda),$$
which is the \emph{multifractal decomposition of $I$ by Lyapunov exponents}.  In particular, we are interested in the Hausdorff dimension of these sets:
$$L(\lambda):=\dim_H(J(\lambda)).$$
This function, usually called \emph{Lyapunov spectrum}, can be computed by means of thermodynamic formalism, in particular using the pressure and its derivative.

To state our second main theorem, we define
\begin{equation*}
A_f:=\begin{cases}
[\lambda_{\inf}, -D^-p(t^+)) & \text{ if } \lambda_m(f)>0,\\
\{0\} & \text{ if } \lambda_m(f)=0.
\end{cases}
\end{equation*}
The following Theorem was obtained in \cite{IomTod11}.  A further theorem regarding the multifractal spectrum of local dimension was also proved there, but we omit an exposition of that for brevity.

\begin{maintheorem}
Let $f\in \F_G$ and $\lambda \in \R\sm A$.  Then
$$L(\lambda)=\frac1\lambda\inf_{t\in \R}(p(t)+t\lambda).$$
If $\lambda \in (-D^-p(t^+), \lambda_M(f))$ then
$$L(\lambda)=\frac1\lambda (p(t_\lambda)+t_\lambda \lambda)=\frac{h(\mu_{t_\lambda})}\lambda.$$
If $\lambda_m(f)>0$ and $\lambda \in A$ then
$$L(\lambda)\ge\frac1\lambda\inf_{t\in \R}(p(t)+t\lambda).$$
Moreover, $\dim_H(J')=1$.
\label{thm:main lyap}
\end{maintheorem}

While the rigorous study of Lyapunov spectrum began with the work of Weiss \cite{we}, previous studies which built up the theory include  \cite{ColLiePor87, EckPro86,HalJenKad86,ra}. Since then the relation established in Theorem \ref{thm:main lyap} between the pressure and the Lyapunov spectrum has been proved in a wide range of different settings. For example, for maps with parabolic fixed points results have been obtained in \cite{bi,GelRam09,Nak00,PolWe}. For rational maps a good description of the Lyapunov spectrum has been given in \cite{GelPrzRam10,gprr}; we note that here, as well as in \cite{GelRam09}, very different methods were used to those presented here.  For a study of this kind of equilibrium state and of Lyapunov exponents in classes of non-uniformly hyperbolic systems in dimension two see \cite{LepRio09, LepOliRio11, SenTak11, SenTak12}. For an example of Theorems~\ref{thm:main eq} and \ref{thm:main lyap} in the case of a flow, see \cite{PacTod10}, where they were proved for contracting Lorenz-like maps, via the results described here applied to the class of cusp maps given below.

\subsection{Extension to other classes of interval maps}
\label{ssec:cusps}
The theory outlined in this paper applies to various classes of piecewise smooth interval/circle maps.  This includes Manneville-Pomeau maps and Lorenz-like maps, but our main focus is on smooth interval maps with critical points, the class $\F$.  However,  we briefly describe a generalisation of this class, so-called \emph{cusp maps}, studied by Ledrappier \cite{Led81} and later extended by Dobbs \cite{Dob08}.  This class includes the class of contracting Lorenz-like maps, see for example \cite{Rov93}.  The extension of our main theorems to this class is described in \cite{IomTod11}.

\begin{defi}
$f:\cup_jI_j\to I$ is a \emph{cusp map} if there exist constants $C,\alpha>1$ and a set $\{I_j\}_j$ is a finite  collection of disjoint open subintervals of $I$ such that
\begin{enumerate}[  itemsep=0.0mm, topsep=0.0mm, leftmargin=7mm]
\item $f_j:=f|_{I_j}$ is $C^{1+\alpha}$ on each $I_j=:(a_j, b_j)$ and $|Df_j|\in (0, \infty)$.

\item $D^+f(a_j), \ D^-f(b_j)$ exist and are equal to 0 or $\pm\infty$.

\item For all $x,y\in \overline{I_j}$ such that $0<|Df_j(x)|, |Df_j(y)|\le 2$ we have $|Df_j(x)-Df_j(y)|<C|x-y|^\alpha$.

\item For all $x,y\in \overline{I_j}$ such that $|Df_j(x)|, |Df_j(y)|\ge 2$, we have $|Df_j^{-1}(x)-Df_j^{-1}(y)|<C|x-y|^\alpha$.
\end{enumerate}
Matching the notation above, denote the set of points $a_j, b_j$ by $\crit$.
\label{def:cusp}
\end{defi}

\begin{rem}
Notice that if for some $j$, $b_j=a_{j+1}$, i.e. $I_j\cap I_{j+1}$ intersect, then $f$ may not continuously extend to a well defined function at the intersection point $b_j$, since the definition above would then allow $f$ to take either one or two values there.  So in the definition above, the value of $f_j(a_j)$ is taken to be $\lim_{x\searrow a_j}f_j(x)$ and $f_j(b_j)=\lim_{x\nearrow b_j}f_j(x)$, so for each $j$, $f_j$ is well defined on $\overline{I_j}$.
\label{rmk:boundaries}
\end{rem}

\begin{rem}
In contrast to the class of smooth maps $\F$ considered previously in this paper, for cusp maps we can have $\lambda_M=\infty$ and/or $\lambda_m=-\infty$.  The first possibility follows since we allow singularities (points where the one-sided derivative is $\infty$).  The second possibility follows from the presence of critical points (although is avoided for smooth multimodal maps with non-flat critical points by \cite{Prz93}).  Examples of both of these possibilities can be found in \cite[Section 11]{Dob08}.
\label{rmk:lyap infinity}
\end{rem}

%We will ultimately be interested in cusp maps without singular points with negative Schwarzian derivative (in fact the latter rules out the former).  Note that since we are only interested in the transitive parts the system, transitive multimodal maps as in the rest of the paper can be considered to fit into this class.

\subsection{Outline of the paper}
The layout of the paper is as follows. In Section \ref{sec:CMS} we give a detailed exposition of thermodynamic formalism in the context of the full-shift on a countable alphabet; examples of some of the irregular behaviour that can occur in this setting are given. Recall that the importance of that section relies on the fact that the inducing schemes we consider are conjugated to the full-shift on a countable alphabet. In Section \ref{sec:prelim dim}  a brief summary of classical definitions and results on dimension theory that will be used in the rest of the paper is given. Section \ref{sec:ind} introduces and describes the induced systems. This is our main technical device and it allows for the proof of all the results presented here. Several important results are given in this section, for example it is explained how the entropy of a measure in the induced system relates to the entropy of the measure on the original system (Abramov's formula) and also how the integral of a potential relates to the integral of its induced version (Kac's formula). Moreover, a fundamental result stating that we only need to consider a countable collection of inducing schemes in order to `see' all the ergodic theory of the original system is given. In Section \ref{sec:app ind} we explain how to implement the inducing approach in order to describe the thermodynamic formalism of the original system: we outline the proof of the relevant result, Theorem \ref{thm:psi ind}. In Section \ref{sec:mfa} we explain the proof of Theorem \ref{thm:main lyap} where we describe the Lyapunov spectrum in the context of multimodal maps.  We also include appendices, containing both explanatory comments and new proofs of some results. For instance in Appendix \ref{sec:Hof} we describe and explain the Hofbauer tower technique which is implicit in our results. In Appendix \ref{sec:appendix} we give a new proof of a result that relates conformal measures for the induced system with conformal measures in the original one. We give conditions that ensure that we can project a conformal measure from the induced system into the original one. Finally, in Appendix \ref{sec:Markdim} we correct an error in a result we previously used from \cite[Proposition 3]{PolWe}, where the pointwise dimension is related to the Markov dimension.  This result is used in our proof of Theorem \ref{thm:main lyap}.

%\subsection{Things to mention}
%\textbf{MT:moved this section to here and deleted the stuff we've dealt with.}

%Maybe should comment on the weaking of the assumption $\sup \phi <P(\phi)$ by Juan and coauthos and how this might stregenth our resulst on existence of eq states

\section{Thermodynamic formalism for countable Markov shifts}
\label{sec:CMS}

This section is devoted to study thermodynamic formalism for countable Markov shifts. While the theory is well developed for topologically mixing systems (see \cite{Sar99}), we will concentrate on a particular case, namely the full-shift, which corresponds to the symbolic representation of the \emph{inducing schemes} that we will consider. In this setting the theory was mostly developed by Mauldin and Urba\'nski \cite{MUifs}.  We will see that the thermodynamic formalism in this setting is similar to that of finite state Markov shifts \cite{Ru_book}.

The \emph{full-shift} on a countable alphabet $(\Sigma, \sigma)$ is the set
\[ \Sigma := \left\{ (x_n)_{n \in \N_0} : x_n \in \N_0 \text{ for every } n \in \N_0 \right\}, \]
together with the shift map $\sigma: \Sigma  \to \Sigma $ defined by
$\sigma(x_0, x_1, \dots)=(x_1, x_2,\dots)$.
The set $C_{i_0 \dots i_{n-1}}:= \{(x_n)_n \in \Sigma : x_0=i_0 \dots x_{n-1}=i_{n-1}    \}$ is called a \emph{cylinder} of length $n$. The space $\Sigma$ endowed  with the topology generated by the cylinder sets is a non-compact space. This fact is one of the main difficulties that need to be addressed to develop the theory. Indeed, the classical
definition of pressure using $(n, \epsilon)$-separated sets (see \cite[Chapter 9]{Waltbook}) depends upon the metric and, if the space is not compact, can be different even for two equivalent metrics.  The regularity assumptions on the potentials are fundamental when it comes to describe and develop the thermodynamic formalism.
 The \emph{$n$-th variation} of $\phi:\Sigma \to \R$ is defined by
\[ V_{n}:= \sup \{| \phi(x)-  \phi(y)| : x,y \in \Sigma, x_{i}=y_{i}, 0 \leq i \leq n-1 \}. \]
Note that $\phi$ is continuous if and only if $V_{n} \to 0$. We say that $ \phi$ is of \emph{summable variations} if $\sum_{n=2}^{\infty} V_n(\phi)< \infty$. We say that it is  \emph{weakly H\"older} (with parameter $\theta$) if there exists $\theta \in (0,1)$ such that for $n \geq 2$ we have  $V_{n}( \phi) =O(\theta^{n}).$   Clearly, if $\phi$ is of summable variations then it is continuous and if $\phi$ is weakly H\"older then it is of summable variations. The following definition was given by Mauldin and Urba\'nski \cite{MUifs}.

\begin{defi}
Let $(\Sigma, \sigma)$ be the full-shift on a countable alphabet and let $\phi : \Sigma \to \R$ a potential of summable variations. The \emph{pressure} of $\phi$ is given by
\begin{equation*}
P(\phi):= \lim_{n \rightarrow \infty} \frac{1}{n} \log \sum_{\sigma^{n}x=x} \exp \left( \sum_{i=0}^{n-1}  \phi(\sigma^{i}x)\right).
\end{equation*}
\end{defi}

The limit always exists but it can be infinity.  Indeed $P(0)= \infty$, which by the Variational Principle, which we'll see later, means that the entropy of the full-shift is infinity.

\begin{rem}
Sarig \cite{Sar99}, building up on work by Gurevich \cite{Gur69,Gur70}, extended this definition to arbitrary topologically mixing countable Markov shifts $(\Sigma, \sigma)$.  He defined the \emph{Gurevich pressure} of a potential $\phi : \Sigma \to \R$ of summable variations by
\begin{equation*}
P_G(\phi):= \lim_{n \rightarrow \infty} \frac{1}{n} \log \sum_{\sigma^{n}x=x} \exp \left( \sum_{i=0}^{n-1}  \phi(\sigma^{i}x) \right)\chi_{C_{i_0}}(x),
\end{equation*}
where $\chi_{C_{i_{0}}}(x)$ is the characteristic function of the cylinder $C_{i_0}$. It can be proved that the limit always exists (it can be infinity) and that it is independent of $i_0$. Moreover, if $(\Sigma, \sigma)$ is the full-shift then the Gurevich pressure coincides with the pressure defined by Mauldin and Urba\'nski \cite{SarBIP}, i.e., $P_G(\phi)=P(\phi)$.
\end{rem}

If the potential $\phi:\Sigma \to \R$ depends only on the first coordinate, that is for every $i\in \N$ there exists $\lambda_i>0$ such that $\phi(x_i x_1 \dots)= \log\lambda_i$, then the pressure can be explicitly computed:
\begin{align*}
P(\phi)&=  \lim_{n \rightarrow \infty} \frac{1}{n} \log\sum_{\sigma^{n}x=x} \exp \left( \sum_{i=0}^{n-1}  \phi(\sigma^{i}x) \right) \\
&= \lim_{n \rightarrow \infty} \frac{1}{n} \log \sum_{(j_0 \dots j_{n-1}) \in \N^n} (\lambda_{j_0} \lambda_{j_1} \cdots \lambda_{j_{n-1}}) \\
&=   \lim_{n \rightarrow \infty} \frac{1}{n} \log \left( \sum_{i \in \N_0}  \lambda_i  \right)^n = \log \sum_{i=0}^{\infty} \lambda_i.
\end{align*}
The pressure satisfies the following approximation property (see \cite{MUifs,Sar99}).

\begin{teo}
Let $(\Sigma, \sigma)$ be the full-shift and $\phi : \Sigma \to \R$ a potential of summable variations. If $\mathcal{K}:= \{ K \subset \Sigma : K \textrm{ compact and } \sigma\textrm{-invariant}, K \neq \emptyset \}$ then
\begin{equation*} \label{aprox}
   P(\phi) = \sup \{ P(\phi|K) : K \in \mathcal{K} \},
\end{equation*}
where $P(\phi| K)$ is the topological pressure of $\phi$ restricted to the compact set $K$ (for definition and properties see \cite[Chapter 9]{Waltbook}).
\label{thm:Gur approx}
\end{teo}
This notion of pressure also satisfies the Variational Principle (see \cite{MUifs,Sar99}):

\begin{teo}[Variational Principle]
Let $(\Sigma, \sigma)$ be the full-shift and $\phi : \Sigma \to \R$ a potential of summable variations then
\begin{equation*}
P(\phi)= \sup \left\{ h(\mu) + \int \phi~d\mu : \mu \in \M_{\sigma} \text{ and } - \int \phi~d\mu < \infty \right\}
\end{equation*}
where $ \M_{\sigma} $ denotes the set of $\sigma-$invariant probability measures and $h(\mu)$ denotes the entropy of the measure $\mu$ (for a precise definition see  \cite[Chapter 4]{Waltbook}).
\end{teo}

It is worth pointing out the remarkable fact that this theorem relates quantities of a different nature. While the pressure is defined in topological terms, the left hand side of the Variational Principle only depends on the Borel structure.  %Note that a definition of pressure that depends upon the metric has the problem that, since the entropy of the measure and the integral do not depend on the metric, might not satisfy the Variational Principle.
Also note that since the system has infinite entropy, it is possible for a measure $\mu \in  \mathcal{M}_{\sigma}$ to have $h(\mu)= \infty$ and $ \int \phi~d\mu = - \infty$. In that case, the sum of both quantities is meaningless, which is why the condition $ - \int \phi~d\mu < \infty $ is imposed.

A measure $\mu \in \M_{\sigma}$ attaining the supremum in the Variational Principle, that is
\begin{equation*}
P(\phi)= h(\mu) + \int \phi~d\mu,
\end{equation*}
 is called an \emph{equilibrium state} for $\phi$. Buzzi and Sarig \cite{BuSar} proved that a potential of summable variations has at most one equilibrium state. There are examples of locally constant potentials of finite pressure that don't have equilibrium states (see Example \ref{ejeGibbs}).

Given a potential $\phi:\Sigma \to \R$, a measure $\mu$ on $\Sigma$ is called a \emph{Gibbs measure} for the potential $\phi$ if there exist constants $C>0$ and $P\in \R$ such that for every cylinder $C_{i_0 i_1  \dots i_{n-1}}$ and every $x \in C_{i_0 i_1  \dots i_{n-1}}$ we have
\begin{equation*}
\frac{1}{C} \leq \frac{\mu(C_{i_0 i_1  \dots i_{n-1}})}{\exp(-nP +  \sum_{i=0}^{n-1} \phi(\sigma^i x))} \leq C.
\end{equation*}
That is, the measure of a cylinder is comparable with the Birkhoff sum of the potential, normalised by some $P$, the \emph{Gibbs constant}. This a useful property that enables us to estimate the measure of a set. In the settings described here, if $\mu$ is an equilibrium state for $\phi$, then it is a Gibbs measure with Gibbs constant $P=P(\phi)$.  The following result was obtained (in a slightly more general setting) by Sarig in \cite{SarBIP}. The sufficient part of the theorem was proved by Mauldin and Urba\'nski \cite{muGIBBS}.

\iffalse
Let $\phi:\Sigma \to \R$ be a potential of finite pressure. A measure $\mu \in \M_{\sigma}$ is called a \emph{Gibbs measure} for the potential $\phi$ if there exists a constant $C>0$ such that for every cylinder $C_{i_0 i_1  \dots i_{n-1}}$ and every $x \in C_{i_0 i_1  \dots i_{n-1}}$ we have
\begin{equation*}
\frac{1}{C} \leq \frac{\mu(C_{i_0 i_1  \dots i_{n-1}})}{\exp(-nP(\phi) +  \sum_{i=0}^{n-1} \phi(\sigma^i x))} \leq C.
\end{equation*}
That is, the measure of a cylinder is comparable with the Birkhoff sum of the potential. This a useful property that enables us to estimate the measure of a set. The following result was obtained (in a slightly more general setting) by Sarig in \cite{SarBIP}. The sufficient part of the theorem was proved by by Mauldin and Urba\'nski \cite{muGIBBS}.
\fi

\begin{teo}
Let $(\Sigma, \sigma)$ be the full-shift and $\phi : \Sigma \to \R$ a potential such that $\sum_{i=1}^{\infty} V_i(\phi) < \infty$ then $\phi$ has an invariant Gibbs measure if and only if $P(\phi) < \infty$.
\label{thm:Gibbs BIP}
\end{teo}

 \begin{eje} \label{ejeGibbs}
 The following example was given by Sarig in \cite{SarBIP}. Let $(\Sigma, \sigma)$ the the full-shift and for $n\in \N$, let  $a_{n}:=1/(2n (\log 2n)^2)$, and set $a_0\in \R$ such that $\sum_{n=0}^{\infty} a_{n}= 1$.  Then define the potential $ \phi : \Sigma \to \R$ by $\phi (x_0 x_1 \dots)= \log a_{x_0}$.  Note that $P(\phi)=0$.  The potential $\phi$ does not have an equilibrium state. Indeed, the corresponding Gibbs measure is the Bernoulli measure $\mu(C_{i_0 \dots i_n})= \prod_{j=0}^n a_{x_j}$. We have that $h(\mu)= \infty$ and $\int \phi \ d \mu = -\infty$. In particular this example shows that, opposite to what happens for finite state Markov shifts, a Gibbs measure might not be an equilibrium state. %Moreover, the potential $\phi$ does not have an equilibrium measure.
  \end{eje}

In \cite{Sar01} Sarig generalises the Ruelle Perron Frobenius Theorem to countable Markov shifts. For $\phi$  a potential of summable variations  define the Ruelle operator defined formally in some space of functions by:
\begin{equation*}
L_{\phi}g(x):= \sum_{\sigma y=x} \exp(\phi (y)) g(y).
\end{equation*}

\begin{teo}
Let $(\Sigma , \sigma)$ be the full-shift and $\phi:\Sigma \to \R$ a potential such that  $\sum_{i=1}^{\infty} V_i(\phi) < \infty$ and $P(\phi)=\log \lambda$. Then there exists a conservative conformal measure $m$ and a continuous function $h$ such that $L_{\phi}^{*} m= \lambda m$,
$L_{ \phi} h= \lambda h$ and  $\int h \ d m < \infty$. The measure $ \mu= h m$ is an equilibrium state for $\phi$ if $ \int \phi \  d \mu < \infty$.
\label{thm:RPF mini}
\end{teo}

\subsection{Possible pressure behaviours}

Given a potential $\phi: \Sigma \to \R$ we will be particularly interested in the pressure function  $t \mapsto P(t \phi)$. It follows directly from the approximation property of the pressure (Theorem \ref{aprox}) and the fact that pressure for  finite state  Markov shifts is convex that the pressure function is convex. Moreover, if the potential $\phi$ is non-positive then the pressure function is non increasing. Actually, if $\phi$ is non-positive there exists $t_0 \in [0,\infty]$
  such that
\begin{equation*}
P(t \phi)=
\begin{cases}
\infty & \text{ if } t < t_0;\\
\text{finite}   & \text{ if } t > t_0.
\end{cases}
\end{equation*}
The regularity of the pressure function is well understood. The following result was proved by Sarig  and by Cyr and Sarig \cite{Sar01a,CyrSar}.

\begin{teo}
Let $(\Sigma , \sigma)$ be the full-shift and $\phi:\Sigma \to \R$ a weakly H\"older potential of finite pressure then, when finite, the pressure function is real analytic and
\begin{equation*}
\frac{d}{dt}P(t \phi) \Big|_{t=t^+} = \int \phi \ d \mu_{t^+},
\end{equation*}
where $\mu_{t^+}$is the equilibrium state corresponding to $t^+ \phi$.
\label{thm:press der}
\end{teo}
Summarising we see that for regular potentials the thermodynamic formalism in the full-shift behaves in similar way as the finite state Markov shifts. A potential does have a corresponding invariant Gibbs measure, and if this measure has finite entropy then it is the unique equilibrium state. The pressure function, when finite, is real analytic.

\begin{eje}
There exist potentials for which the pressure function is either infinite or strictly negative. The class of potentials with that property was called  \emph{irregular} by  Mauldin and Urba\'nski \cite{MUifs}. Let $N \in \N$ be such that the sequence $a_{n}=1/(2(n+N) (\log 2(n+N))^2)$ satisfies $\sum_{n=1}^{\infty} a_n <1$. Then the potential $\phi (x_0 x_1 \dots)= \log a_{x_0}$ is such that
\begin{equation*}
P(t \phi)=
\begin{cases}
\text{infinite} & \text{ if } t < 1;\\
\text{negative} & \text{ if } t \geq 1.
\end{cases}
\end{equation*}
The following example was given by Mauldin and Urba\'nski \cite{MUifs}. For each $n \in \N$ consider $2^{n^2-1}+1$ cylinders of length one (partition the space in cylinder of length one according to this criteria). Define a locally constant potential $\psi(x)= 2^{-(n^2+n)}$, if $x$ belongs to one of the above cylinders. Then
\begin{equation*}
P(t \psi)=
\begin{cases}
\text{infinite} & \text{ if } t < 1;\\
\text{negative} & \text{ if } t \geq 1.
\end{cases}
\end{equation*}
\end{eje}

\begin{eje}
The following is an example of a potential which is not irregular. Let $a_n= n(n+1)$ and consider $\phi:\Sigma \to \R$ to be the locally constant potential defined by $\phi(x) \big|_{C_n}:= -\log a_n$. Then
$$P(t \phi) =\log \sum_{n=1}^{\infty} \left( \frac{1}{n(n+1)}  \right)^t.$$
Therefore
\begin{equation*}
P(t \psi)=
\begin{cases}
\text{infinite} & \text{ if } t < 1/2;\\
\text{real analytic} & \text{ if } t \geq 1/2.
\end{cases}
\end{equation*}
In this case we have that
$$ \lim_{t \to 1/2} P(t \phi) = \infty.$$
\end{eje}
As we have seen,  for a non-positive potential $\phi$ there exists a critical  value $t_0 \geq 0$ such that for $t <t_0$ the pressure function is infinite and for $t > t_0$ is finite. If $P(t_0 \phi) < \infty$ it is possible for
$$\lim_{t \to t_0^+} P'(t \phi)$$
to be either finite or infinite.

\begin{eje}
The following examples are due to Kesseb\"ohmer, Munday and  Stratmann \cite{kms}. Let $a_n := n^{-2} (\log(n + 5))^{-12}/C$, where
$C :=\sum_{n=1}^{\infty} n^{-2} (\log(n + 5))^{-12}$. Consider $\phi:\Sigma \to \R$ to be the locally constant potential defined by $\phi(x) \big|_{C_n}:= \log a_n$.
In this case $t_0= 1/2$ and $P((1/2) \phi) < 1$. Moreover,
$$\lim_{t \to 1/2^+} P'(t \phi) < \infty.$$
Consider now  $b_n := n^{-2} (\log(n + 5))^{-4}/C$, where
$C :=\sum_{n=1}^{\infty} n^{-2} (\log(n + 5))^{-4}$. Consider $\psi:\Sigma \to \R$ to be the locally constant potential defined by $\psi(x) \big|_{C_n}:= \log a_n$.
In this case $t_0= 1/2$ and $P((1/2) \psi) < 1$ and
$$\lim_{t \to 1/2^+} P'(t \psi) = \infty.$$
\end{eje}

 %----------------------
\section{Dimension theory}\label{sec:prelim dim}
%----------------------

This section is devoted to recalling basic definitions and results from dimension theory that will be used in what follows (see \cite{Pesbook} and \cite{PrzUrb_book} for details). A countable collection of sets $\{U_i \}_{i\in \N}$ is called a $\delta$-cover of $F \subset\R$ if $F\subset\bigcup_{i\in\N} U_i$, and $U_i$ has diameter $|U_i|$ at most $\delta$ for every $i\in\N$. Letting $s>0$, we define
\[
\mathcal{H}^s(J) := \lim_{\delta \to 0}\inf \left\{ \sum_{i=1}^{\infty} |U_i|^s : \{U_i \}_i \text{ a } \delta\text{-cover of } J \right\}.
\]
The \emph{Hausdorff dimension} of the set $J$ is defined by
\[
{\dim_H}(J) := \inf \left\{ s>0 : \mathcal{H}^s(J) =0 \right\}.
\]

Hausdorff dimension is invariant under bi-Lipschitz transformations
\begin{prop}
Let $\pi: J \subset \R^n \to \R^n$ be a bi-Lipschitz transformation then
\[\dim_H(J)= \dim_H(\pi(J)).\]
\end{prop}

Given a finite Borel measure $\mu$ in $F$, the \emph{lower pointwise dimension}  and \emph{upper pointwise dimension} of $\mu$ at the point $x$ are defined by
$$\underline d_\mu(x):=\liminf_{r\to 0}\frac{\log\mu(B(x, r))}{\log r} \quad \text{ and } \quad \overline d_\mu(x):=\limsup_{r\to 0}\frac{\log\mu(B(x, r))}{\log r},$$
respectively, where $B(x,r)$ is the ball at $x$ of radius $r$.  Whenever these limits are equal, we denote their common limit by $d_\mu(x)$, the  \emph{pointwise dimension} of $\mu$ at $x$.
This function describes the power law behaviour of
$\mu(B(x,r))$ as $r \rightarrow 0$, that is
\[\mu(B(x,r)) \sim r^{d_{\mu}(x)}.\]
The pointwise dimension quantifies how concentrated a measure is around a point: the larger it is  the less concentrated the measure is around that point.
Note that if $\mu$ is an atomic measure supported at the point $x_{0}$ then $d_{\mu}(x_{0})=0$ and if $x_{1} \neq x_{0}$ then $d_{\mu}(x_{1}) = \infty $.

The following propositions relating the pointwise dimension with the Hausdorff dimension can be found in \cite[Section 7]{Pesbook}.

\begin{prop}\label{prop:upper}
Given a finite Borel measure $\mu$, if $\underline d_{\mu}(x) \le d$ for every $x \in F$, then ${\dim_H}(F)\le d$.
\end{prop}

The \emph{Hausdorff dimension} of the measure $\mu$ is defined by
\[
{\dim_H}(\mu) := \inf \left\{ {\dim_H}(Z): \mu(Z)=1 \right\}.
\]

\begin{prop}\label{prop:lower}
Given a finite Borel measure $\mu$, if $d_{\mu}(x) =d$ for $\mu$-almost every $x \in F$, then ${\dim_H}(\mu)=d$.
\end{prop}

We finish this section by defining another notion of dimension, which applies whenever we have a dynamical system $f:X\to X$ and a partition $\P$, which thus gives rise to the notion of $k$-cylinders.  Denote the $k$-cylinder at $x\in X$ by $I_k(x)$.  Note that $x$ could be in two $k$-cylinders simultaneously, but in this case we make an arbitrary choice: in fact these points will not be important here due to their small dimension and measure.  In the case of a shift, this issue does not arise, and $\P$ can be taken as the partition into 1-cylinders.

\begin{defi}
Given a system as above with a probability measure $\mu$ on $X$, we define the \emph{lower Markov pointwise dimension} and \emph{upper Markov pointwise dimension} of $\mu$ at the point $x$ as
$$\underline\delta_{\mu}(x) := \liminf_{k \to \infty} \frac{\log \mu(I_k(x))}{\log |I_k(x)|} \quad \text{ and } \quad \overline\delta_{\mu}(x) := \limsup_{k \to \infty} \frac{\log \mu(I_k(x))}{\log |I_k(x)|},$$
respectively. Here, $| \cdot|$ denotes the euclidean length. If these values coincide, then their common value, the \emph{Markov pointwise dimension} of $\mu$ at $x$, is denoted $\delta_{\mu}(x)$.
\end{defi}

\section{Inducing schemes}
\label{sec:ind}
In this section we introduce our main technical tool, namely \emph{inducing schemes}. We show how its ergodic theory is related to that of the original system and how can it be understood using thermodynamic formalism for countable Markov shifts. A natural way of thinking of inducing schemes is as a generalisation of the first return time map.

We say that $\left(X,\{X_i\}_i, F,\tau\right)=(X,F,\tau)$ is an \emph{inducing scheme} for $(I,f)$ if
\begin{list}{$\bullet$}{\itemsep 0.2mm \topsep 0.2mm \itemindent -0mm \leftmargin=5mm}
\item $X$ is an interval containing a finite or countable
collection of disjoint intervals $\{X_i\}_i$ \st $F$ is hyperbolic and maps each $X_i$
diffeomorphically onto $X$, with bounded distortion (i.e. there exist $C>0$ and $\lambda>1$ such that $|DF^n|\ge C\lambda^n$, and $K>0$ so that for all $i$ and $x,y\in X_i$, $\left| DF(x)/DF(y)-1\right| \le K|x-y|$);
\item $\tau|_{X_i} = \tau_i$ for $\tau_i \in \N$ and $F|_{X_i} = f^{\tau_i}$.  If $x \notin \cup_iX_i$ then $\tau(x)=\infty$.
\end{list}

For our main theorems to hold, we don't need such a strong bounded distortion condition, but for other applications we often do, so we leave the condition in here.
The function $\tau:\cup_i X_i \to \N$ is called the {\em inducing time}. It may
happen that $\tau(x)$ is the first return time of $x$ to $X$, but
that is certainly not the general case.   We denote the set of points $x\in I$ for which there exists $k\in \N$ such that $\tau(F^n(f^k(x)))<\infty$ for all $n\in \N$ by $(X,F,\tau)^\infty$.

The space of $F$-invariant measures is related to the space of $f$-invariant measures. Indeed, given an $f$-invariant measure $\mu$, if there is an $F$-invariant measure $\mu_F$ such that for a subset $A\subset I$,
\begin{equation} \label{eq:ind-meas}
\mu(A)= \frac{1}{\int\tau~d\mu_F}\sum_{i} \sum_{k=0}^{\tau_i-1} \mu_F \left( f^{-k}(A) \cap X_i \right)
\end{equation}
where $\frac{1}{\int\tau~d\mu_F}<\infty$,
we call $\mu_F$ the \emph{lift} of $\mu$ and say that $\mu$ is a \emph{liftable} measure. Conversely, given a measure $\mu_F$ that is $F$-invariant we say that $\mu_F$ \emph{projects} to $\mu$ if \eqref{eq:ind-meas} holds. We say that an $f$-invariant probability measure $\mu$ is \emph{compatible} with the inducing scheme $(X,F, \tau)$ if
\begin{itemize}
\item  $\mu(X)>0$ and $\mu(X \setminus (X, F)^{\infty})=0$, and
\item there exists a $F$-invariant measure $\mu_F$ which projects to $\mu$
\end{itemize}

For a potential $\phi:I\to \R$, $x\in I$ and $k\in \N$, we define
$$S_k\phi(x):=\phi(x)+\phi\circ f(x)+\cdots + \phi\circ f^{k-1}(x).$$
Given an inducing scheme $(X,F, \tau)$, we define the induced potential $\Phi:X\to \R$ by
$\Phi(x)=S_{\tau(x)}\phi(x)$ for $x\in X$ whenever $\tau(x)<\infty$, and $-\infty$ otherwise.

Let $\mu$ be a liftable measure and  be $\nu$ be its lift. A classical result by Abramov \cite{Abr59} (see also \cite{PesSen08, Zwe05})  allows us to relate the entropy of both measures. Further results obtained in  \cite{PesSen08, Zwe05} allow us to do the same with the integral of a given potential  $\phi:I\to \R$.   Indeed, for the induced potential $\Phi$ we have that
\begin{equation*}
h(\mu)=\frac{h(\nu)}{\int \tau \ d \nu}  \textrm{ and  } \int \phi \ d \mu= \frac{\int \Phi \ d \nu}{\int \tau \ d \nu}.
\end{equation*}

The following result appeared in \cite[Theorem 3.3]{IomTod11}  (it can be proved using \cite{Tod}). It provides us with a countable family of relevant inducing schemes.

\begin{teo}\label{thm:schemes}
Let $f\in \F$.  There exist a countable collection $\{(X^n,F_n)\}_n$ of inducing schemes  such that:

\begin{enumerate}[label=({\alph*}),  itemsep=0.0mm, topsep=0.0mm, leftmargin=7mm]
\item any ergodic invariant probability measure $\mu$ with $\lambda(\mu)>0$ is compatible with one of the inducing schemes $(X^n, F_n)$.  In particular there exists an ergodic $F_n$-invariant probability measure $\mu_{F_n}$ which projects to $\mu$;
\item any equilibrium state for $-t\log|Df|$ where $t\in \R$ with $\lambda(\mu)>0$, or for a H\"older continuous potential $\phi:I \to \R$ with $\sup \phi<P(\phi)$, is compatible with all inducing schemes $(X^n, F_n)$.
\item if $f\in \F_G$ then $$dim_H\Big(I\sm\left(\cup_{n=1}^\infty(X^n, F_n)^\infty\right) \Big)=0.$$ \end{enumerate}
\end{teo}

 To give a trivial example of part of the importance of this theorem, suppose that a point $x_0\in I$ is a repelling periodic point for $f$: there exists $p\ge 1$ such that $f^p(x_0)=x_0$ and $|Df^p(x_0)|>1$.  Then for a small interval $X\ni x_0$ there is a subset $X_1\ni p$ such that $f^p(X_1)=X$.  So we have produced an inducing scheme $(X, F)$ with one branch $X_1$ and $\tau|_{X_1}=p$.  The measure $\mu:=\frac1p(\delta_{x_0}+
\delta_{f(x_0)}+\cdots + \delta_{f^{p-1}(x_0)})$ lifts to $(X, F)$, but is the unique measure which does so.  Therefore, if we were only interested in $\mu$, then $(X, F)$ is sufficient for our purposes, but if we wanted to consider other measures, for example with positive entropy, then $(X, F)$ is useless.

\subsection{Young towers}

Observe that an inducing scheme $(X,F,\tau)$ for $(I, f)$ can be used to build a tower, which in the context described here is often referred to as a Young tower (see eg \cite{You99}).  Even though we don't use these ideas here directly, we give a short description to clarify its relationship with the inducing scheme.  The tower is defined as the disjoint union
\[
\Delta = \bigsqcup_{i} \bigsqcup_{j = 0}^{\tau_i-1}
(X_i,j),
\]
with dynamics
\[
f_\Delta(x,j) = \left\{ \begin{array}{ll}
(x,j+1) & \mbox{ if } x \in X_i, j < \tau_i-1; \\
(F(x),0) & \mbox{ if } x \in X_i, j = \tau_i-1. \\
\end{array} \right.
\]
Given $i$ and  $0\le j<\tau_i$, let $\Delta_{i,j} :=\{(x,j):
x\in X_i\}$ and $\Delta_l:=\bigcup_{i}\Delta_{i,l}$ is
called the {\em $l$-th floor}. Define the natural projection $\pi_\Delta:\Delta \to X$ by $\pi_\Delta(x,j) =
f^j(x)$, and $\pi_X:\Delta \to X$ by $\pi_X(x,j) = x$.
Note that $(\Delta, f_\Delta)$ is a Markov system, and the first return map of $f_\Delta$ to the \emph{base} $\Delta_0$ is isomorphic to $(X,F,\tau)$.

Also, given $\psi:I\to \R$, let $\psi_\Delta:\Delta\to \R$ be defined by $\psi_\Delta(x,j) =\psi(f^j(x))$.  Then the induced potential of $\psi_\Delta$ to the first return map to $\Delta_0$ is exactly the same as the induced potential of $\psi$ to the inducing scheme $(X,F,\tau)$.

\subsection{Coding}

One of the main tools in the proof of the results presented here is that an inducing scheme as described here can be coded by the full-shift on at most countably many symbols.   That is, given a point $x\in (X, F)^\infty$, this point is given the code $\pi x=(x_0, x_1, \ldots)$ where $x_k=i$ if $F^k(x)\in X_i$.  Notice that $\pi$ is bijective due to the hyperbolicity of $F$.  Thus there is a conjugacy
\begin{equation*}
\begin{CD}
(X,F)^\infty @>F>> (X,F)^\infty\\
@V{\pi}VV	@VV{\pi}V\\
\Sigma @>\sigma>>\Sigma
\end{CD}
\end{equation*}

Moreover, the potentials $\Phi:(X, F) \to [-\infty, \infty]$ have symbolic versions $\Phi\circ \pi^{-1}:\Sigma \to [-\infty, \infty]$.  We abuse notation by not explicitly distinguishing between the symbolic and standard version of potentials.   Observe that if $\Phi:\cup_iX_i\to [-\infty, \infty]$ is H\"older, then the hyperbolicity and distortion conditions on the inducing scheme imply that the symbolic version is also weakly H\"older continuous.  Moreover the properties of our inducing schemes imply that the potential $-\log|DF|$ is also H\"older continuous and its symbolic version is  weakly H\"older continuous.

\subsection{Induced level sets} \label{ind-level}

Given an inducing scheme $(X, F, \tau)$, and $\lambda\in \R$, let
\[J_F(\lambda) := \left\{ x \in X : \lim_{k \to \infty} \frac{\sum_{j=0}^{k-1} \log |DF(F^j(x))|}{\sum_{j=0}^{k-1} \tau(F^j(x))} = \lambda \right\}. \]
A key observation here is that if $x\in (X, F)^\infty$, then  $x\in J(\lambda)$ if and only if  $x\in J_F(\lambda)$.  Moreover, the same holds if $f^k(x)\in (X, F)^\infty$ for some $k\in \N$.  Therefore the Lyapunov spectrum of $(I, f)$ can be completely described via inducing schemes so long as inducing schemes cover a sufficient amount of points in $J(\lambda)$: that is, so long as there is no subset of $J(\lambda)$ of larger Hausdorff dimension than $J_F(\lambda)$ for all our inducing schemes $(X, F, \tau)$.

\section{Implementing the inducing approach}
\label{sec:app ind}

In this section we take the potential $\phi_t-p_f(t)=-t\log|Df|-p_f(t)$ and consider its induced version $\Psi_t=-t\log|DF|-\tau p_f(t)$.  Furthermore, for $i\in \N$, let $\Psi_{i,t}:=\sup_{x\in X_i}\Psi_t(x)$.
The principal result guaranteeing that we can study our measures and multifractal properties using inducing schemes is the following.

\begin{teo}
Given $f\in \F_T$, there exists an inducing scheme $(X, F,\tau)$
such that for each $t\in (-\infty, t^+)$,
\begin{enumerate}[label=({\alph*}),  itemsep=0.0mm, topsep=0.0mm, leftmargin=7mm]
\item $P(\Psi_t)=0$;
\item $\sum_i\tau_i e^{\Psi_{i, t}}<\infty$.
\end{enumerate}
\label{thm:psi ind}
\end{teo}

We sketch the proof of this below, but first state and prove a corollary.

\begin{coro}
Under the conditions of Theorem~\ref{thm:psi ind},
\begin{enumerate}[label=({\roman*}),  itemsep=0.0mm, topsep=-2.5mm, leftmargin=7mm]
\item there exists an equilibrium state $\mu_t$ for $\phi_t$;
\item $t\mapsto P(\phi_t)$ is $C^1$ in $(-\infty, t^+)$;
\item there exists a $\psi_t$-conformal measure $m_t$, and $\mu_t\ll m_t$.
\end{enumerate}
\end{coro}

\begin{proof}
The existence of an equilibrium state $\mu_{\Psi_t}$ for $\Psi_t$ follows from Theorem~\ref{thm:Gibbs BIP} and Theorem~\ref{thm:psi ind}(a).  Since  $\mu_{\Psi_t}$ is a Gibbs measure, to check that it projects to the original system, which is to check that $\int\tau~d\mu_{\Psi_t}<\infty$, follows from Theorem~\ref{thm:psi ind}(b), proving (i).

The fact that $s\mapsto \int\tau~d\mu_{\Psi_s}$ and $s\mapsto \int\log|DF|~d\mu_{\Psi_s}$ are continuous for $s\in (t-\eps, t+\eps)$ can be easily deduced, so long as $\mu_{\Psi_s}$  all lift to the same inducing scheme.  Therefore $s\mapsto \lambda(\mu_s)$ is continuous.  The proof of (ii) then follows from an analogue of Theorem~\ref{thm:press der}, which is standard.

The proof of the existence of a $\Psi_t$-conformal measure, which is equivalent to $\mu_{\Psi_t}$ follows from Theorem~\ref{thm:psi ind}(a) and Theorem~\ref{thm:RPF mini}.  The fact that it projects to a $\psi_t$-conformal measure follows as in Theorem~\ref{thm:proj conf}.
\end{proof}

The main steps in the proof of Theorem~\ref{thm:psi ind} are contained in Sections 4-6 of \cite{IomTod10}.

\begin{enumerate}[  itemsep=0.0mm, topsep=-2.5mm, leftmargin=7mm]
\item $P(\Psi_t)=0$:
\begin{enumerate}[label=({\roman*}),  itemsep=0.0mm, topsep=-2.5mm, leftmargin=7mm]
\item $P(\Psi_t)\le 0$ is straightforward,  see \cite[Lemma 4.1]{IomTod10}.  The idea is to assume the contrary.  Then Theorem~\ref{thm:Gur approx} implies that there is an approximation of $(X, F)$ by a subsystem with $N$ branches with $h(\nu_N)+\int\Psi_t>0$.  Projecting $\nu_N$ to $\mu_N$, Abramov's formula implies that $h(\mu_N)+\int\psi_t~d\mu_N> 0$, contradicting the  Variational Principle.
\item Proving $P(\Psi_t)\ge 0$ is much harder.  The strategy is to find a sequence $(\mu_n)_n$ such that $h(\mu_n)+\int\psi_t~d\mu_n\to 0$, but crucially to show that $\int\tau~d\mu_n$ is uniformly bounded.  The idea for this is that by our choice of $t\in (-\infty, t^+)$, there is a constant $K>0$ such that for all large $n$, $h(\mu_n)\ge K$.  This means that there is a finite set of domains in the Hofbauer extension (see Appendix~\ref{sec:Hof},  particularly Theorem~\ref{thm:cusp facts}) such that each $\mu_n$ gives definite mass to these domains.  By a compactness argument, this gives a domain $\hat X$ such that a subsequence of these lifted measures give $\hat X$ mass $\ge \eps>0$.  The inducing scheme is a first return map in the Hofbauer tower, so Kac's Lemma ultimately gives $\int\tau~d\nu_{n_k}\le \frac1\eps$ where $(\nu_{n_k})_k$ are the lifted versions of these measures, hence giving a uniform bound.  Moreover, Abramov's formula implies that $h(\nu_{n_k})+\int\Psi_t~d\nu_{n_k}\to 0$, so by the Variational Principle $P(\Psi_t)\ge 0$.
\end{enumerate}
\item The existence of the equilibrium state $\mu_{\Psi_t}$ then follows from Theorem~\ref{thm:Gibbs BIP}.  To project this, we need $\int\tau~d\mu_{\Psi_t}<\infty$.  This follows since we can show that the measures $\nu_{n_k}$ above converge to $\mu_{\Psi_t}$.
\item The fact that the inducing scheme works for all $t\in (-\infty, t^+)$ follows from an argument on the Hofbauer tower which says if we can't pass mass from an inducing scheme $(X, F)$ to which $\mu_{\psi_t}$ lifts, to one $(X', F')$  to which $\mu_{\psi_{t'}}$ lifts, then the lifted measure $\nu_{\Psi_{t'}}$ doesn't give mass to all branches of $(X', F')$, which is false.
\end{enumerate}

\section{Multifractal analysis: The Lyapunov spectrum}
\label{sec:mfa}
This Section is devoted to explaining and commenting on the proof of Theorem B.

We will first prove that if $\lambda \in \R \setminus A$ then  $L(\lambda)=\frac1\lambda\inf_{t\in \R}(p(t)+t\lambda)$. In this context our proof relies on the fact that we can an construct equilibrium state $\mu_{\lambda}$ such that $ \lambda(\mu_{\lambda}) = \lambda.$

{\bf{Proof of the lower bound.}}
Let $\lambda \in (-D^-p(t^+), \lambda_M)$. This part of the proof goes very much along the lines of the usual hyperbolic theory (see \cite[Chapter 7]{Pesbook}). Consider the equilibrium measure $\mu_{\lambda}$ corresponding to $-t_{\lambda} \log |Df|$ such that $ \lambda(\mu_{\lambda}) = \lambda. $ We have
\begin{enumerate}
\item $\mu_{\lambda}(I\sm J(\lambda))= 0$;
\item the measure $\mu_{\lambda}$ is ergodic;
  \item  by \cite{Hofdim}, the pointwise dimension is $\mu_{\lambda}$-almost everywhere equal to
\[\lim_{r \to 0} \frac{\log \mu_{\lambda}(B(x,r))}{\log r}   = \frac{h(\mu_{\lambda})}{\lambda}.\]
\end{enumerate}
Therefore, Proposition~\ref{prop:lower} implies
\[ \dim_H(J(\lambda)) \ge  \frac{h(\mu_{\lambda})}{\lambda}=\frac1\lambda\inf_{t\in \R}(p(t)+t\lambda),\]
where the final equality follows from the fact that $t\mapsto p(t)$ is $C^1$ and strictly convex in the relevant domain.

{\bf{Proof of the upper bound.}}  Again, let $\lambda \in (-D^-p(t^+), \lambda_M)$.  The situation in this case is a bit more subtle, we will make use of the induced systems and of its Markov structure. Initially, rather than considering pointwise dimension , we consider a local dimension that is adapted to the Markov structure, namely the Markov dimension defined at the end of Section~\ref{sec:prelim dim}. In Appendix \ref{sec:Markdim} we show that under suitable conditions, these two  local dimensions coincide. This will allow us to obtain an upper bound for the induced level set. In order to prove that this same bound also holds for the original level sets we will make use of the bi-Lipschitz property of the projection map.

Let $(X^n, F_n, \tau)$ be an induced system as constructed in Theorem \ref{thm:schemes} and $J_{F_n}(\lambda)$ be the induced level set defined in Section \ref{ind-level}.
Note that if $\mu_t$ is the equilibrium measure for $-t \log |Df|$ and $\lambda(\mu_t) = \lambda$ then the lifted measure $\mu_{F_n,t}$ has
\[\mu_{F{_n},t} \left( I \sm J_{F_n}(\lambda)   \right)  = 0 .\]

Denote by $I_k^n(x)$ the cylinder (with respect to the Markov dynamical system $(X^n,F_n)$) of length $k$ that contains the point $x \in X$, and by
$|I_k^n(x)|$ its Euclidean length. By definition there exists a positive constant $K\ge1$ such that for every $x \in X$  and every $k \in \mathbb{N}$ we have
\[ \frac{1}{K} \le \frac{|I_k^n(x)|}{|DF_n^k(x)|} \le K.        \]

A simple calculation (see \cite[Lemma 4.3]{IomTod11}) then shows that for every $x\in J_{F_n}(\lambda)$,  $\delta_{\mu_{F_n,t }}(x)=\frac{h(\mu_t)}{\lambda}$.  As proved in Appendix~\ref{sec:Markdim} (see also \cite[Proposition 3]{PolWe}), if $\delta_{\mu_{F_n,t} }(x)$ and $\lambda(x)$ exist then
\[d_{\mu_{F_n,t}}(x) = \delta_{\mu_{F_n,t}}(x).\]
Therefore Proposition~\ref{prop:upper} implies that
\[\dim_H(J_{F_n}(\lambda) )\le \frac{h(\mu_t)}{\lambda}.\]
%These arguments are summarised in the following lemma.

%\begin{lema}
%The Hausdorff dimension of $J_{F_n}(\lambda)$ is given by
%\[\dim_H(J_{F_n}(\lambda) ) = \frac{h(\mu_t)}{\lambda}=\delta_{\mu_{F_n,t }}(x)\]
%for $\mu_{F_n,t }$-a.e. $x\in X^n$.
%\label{lem:Mark pw dim}
 %\end{lema}

\iffalse
The proof of this Lemma relies on the fact that the measure $\mu_{F_n,t }$ is Gibbs with respect to the potential $\Phi_{t,F_n}:=-t \log|DF_n| -P(-t \log |Df|) \tau$. This permits the explicit computation of the limit defining the Markov pointwise dimension. So we obtain
\begin{equation*}
\delta_{\mu_{F_n,t} }(x) = t + \frac{h(\mu_t) -t \lambda}{\lambda}= \frac{h(\mu_t)}{\lambda}.\end{equation*}
As proved in Appendix \ref{sec:Markdim} (see also \cite[Proposition 3]{PolWe}) we have that if $\delta_{\mu_{F_n,t} }(x)$ and $\lambda(x)$ exists then
\[d_{\mu_{F_n,t}}(x) = \delta_{\mu_{F_n,t}}(x).\]
Since  $\mu_{F_n,t}(X^n \sm J_n(\lambda))=1$ we have that
\[\dim_H(J_n(\lambda) )= \frac{h(\mu_t)}{\lambda},\]
as required.
\fi

The definition of $J(\lambda)$ and Theorem~\ref{thm:schemes}(c) implies that up to a set of Hausdorff dimension zero, any $x\in J(\lambda)$ has $k,n\in \N$ such that $f^k(x)\in J_{F_n}(\lambda)$, and since moreover the projection map $\pi: X_n \to I$  is bi-Lipschitz,
\begin{eqnarray*}
\dim_H(J(\lambda))
 \le \dim_H\left(\cup_n\cup_{k\ge 0}f^{-k}\left(\pi_n (J_{F_n}(\lambda))\right)\right) = \\ \sup_n\left\{\dim_H\left(\pi_n (J_{F_n}(\lambda))\right)\right\} =
\frac{h(\mu_t)}{\lambda},\end{eqnarray*}
thus completing the proof of the upper bound in the case $\lambda \in (-D^-p(t^+), \lambda_M)$.

The situation can be much more complicated if $\lambda \in (0, -D^-p(t^+))$ due to the lack of natural equilibrium states (note that the case $\lambda=-D^-p(t^+)$ follows as above, provided $-D^-p(t^+)>0$). However, we are able to prove the following

\begin{lema} \label{lem:small LE big dim}
Given $f\in \F_G$, for any  $\lambda \in (\lambda_m, -D^-p(t^+))$ and $\eps>0$ there exists an ergodic measure $\mu\in \M$ with $\lambda(\mu)=\lambda$ and $\dim_H(\mu)\ge t^++\frac{p_f(t^+)}\lambda-\eps$.
\end{lema}

The proof is obtained approximating $(I,f)$ by hyperbolic sets on which we have equilibrium states with relatively small Lyapunov exponent and large Hausdorff dimension.  The hyperbolic sets are invariant sets for truncated inducing schemes.  For the particular case of $\lambda_m(f)>0$ this covers the statement of Theorem~\ref{thm:main lyap} for $\lambda \in A$. For the particular case that $\lambda_m(f)=0$, since $t^+=1$ and $p_f(t^+)=0$, such a choice of $\lambda$ gives a measure with $\dim_H(\mu)\ge1$.

The same approximation argument together with the result of Barreira and Schmeling \cite{BarSc} allows us to prove that the irregular set $J'$ has full Hausdorff dimension.

\section*{Acknowledgements}

The authors would like to thank  H.\ Bruin, K. \ Gelfert, L.\ Olsen, I.\ Petrykiewicz and Juan Rivera-Letelier for his  useful comments.

\appendix

\section{Hofbauer extension}
\label{sec:Hof}

In this section we describe the Hofbauer extension, the construction of which underlies many of the results presented here. We will not explain our use of its properties in great detail for the sake of brevity.  The setup we present here can be applied to general dynamical systems, since it only uses the structure of dynamically defined cylinders.  An alternative way of thinking of the Hofbauer extension specifically for the case of multimodal interval maps, which explicitly makes use of the critical set, is presented in \cite{BruBru04}.

We let $\cyl_n[x]$ denote the member of $\P_n$, defined as above, containing $x$.  If $x\in \cup_{n\ge 0}f^{-n}(\crit)$ there may be more than one such interval, but this ambiguity will not cause us any problems here.

The \emph{Hofbauer extension} is defined as $$\hat
I:=\bigsqcup_{k\ge 0}\bigsqcup_{\cyl_{k}\in \P_{k}}
f^k(\cyl_{k})/\sim$$ where $f^k(\cyl_{k})\sim
f^{k'}(\cyl_{k'})$ as components of the disjoint union $\hat I$ if $f^k(\cyl_{k})= f^{k'}(\cyl_{k'})$ as subsets in $I$.  Let
$\D$ be the collection of domains of $\hat I$ and   $\hat\pi:\hat
I \to I$ be the natural inclusion map.  A point $\hat x\in \hat I$ can
be represented by $(x,D)$ where $\hat x\in D$ for $D\in \D$ and
$x=\hat\pi(\hat x)$.  Given $\hat x\in \hat I$, we can denote the domain $D\in \D$ it belongs to by $D_{\hat x}$.

The map $\hat f:\hat I \to \hat I$ is defined by
$$\hat f(\hat x) = \hat f(x,D) = (f(x), D')$$
if there are cylinder sets $\cyl_k \supset \cyl_{k+1}$ \st $x \in
f^k(\cyl_{k+1}) \subset f^k(\cyl_{k}) = D$ and $D' = f^{k+1}
(\cyl_{k+1})$.
In this case, we write $D \to D'$, giving $(\D, \to)$ the
structure of a directed graph.  Therefore, the map $\hat\pi$
acts as a semiconjugacy between $\hat f$ and $f$: $$\hat\pi\circ \hat
f=f\circ \hat\pi.$$
We denote the `base' of $\hat I$, the copy of $I$ in $\hat I$, by  $D_0$.  For $D\in \D$, we define $\level(D)$ to be the length of the shortest path $D_0 \to \dots \to D$ starting at the base $D_0$.  For each $R \in \N$, let $\hat I_R$ be the compact
part of the Hofbauer tower defined by
$$
\hat I_R := \sqcup \{ D \in \D : \level(D) \le R \}.$$

For maps in $\F$, we can say more about the graph structure of $(\D, \to)$ since Lemma 1 of \cite{BruTod09} implies that if $f\in \F$ then there is a closed primitive subgraph $\D_{\T}$ of $\D$.  That is, for any $D,D' \in\D_{\T}$ there is a path $D\to \cdots \to D'$; and for any $D\in \D_{\T}$, if there is a path $D\to D'$ then $D'\in \D_{\T}$ too.  We can denote the disjoint union of these domains by $\hat I_{\T}$.  The same lemma says that if $f\in \F$ then $\hat\pi(\hat I_{\T})=\Omega$, the non-wandering set and $\hat f$ is transitive on $\hat I_{\T}$.  Theorem~\ref{thm:cusp facts} below gives these properties for transitive cusp maps.

Given an ergodic measure $\mu\in \M_f$, we say that $\mu$ \emph{lifts to $\hat I$} if there exists an ergodic $\hat f$-invariant probability measure $\hat\mu$ on $\hat I$ such that $\hat\mu\circ\hat\pi^{-1}=\mu$.  For $f\in \F$, if $\mu\in \M_f$ is ergodic and $\lambda(\mu)>0$ then $\mu$ lifts to $\hat I$, see \cite{Kel89, BruKel98}.

Property $(\dag)$ is that for any $\hat x, \hat y\notin \bd\hat I$ with $\hat\pi(x)=\hat \pi(y)$ there exists $n$ such that $\hat f^n(\hat x)=\hat f^n(\hat y)$.  This property holds for cusp maps by the construction of $\hat I$ using the partition $\{I_i\}_i$ given in Definition~\ref{def:cusp}.

We only use the following result in the context of equilibrium states for cusp maps with no singularities.  However, for interest we state the theorem in greater generality.

\begin{teo}
Suppose that $f:I \to I$ is a transitive cusp map with topological entropy $h_{top}(f)>0$.  Then:
\begin{enumerate}
\item there is a transitive part $\hat I_{\T}$ of the tower such that $\hat\pi(\hat I_{\T})=I$;
\item any measure $\mu\in \M_f$ with $0<\lambda(\mu)<\infty$ lifts to $\hat\mu$ with $\mu=\hat\mu\circ \hat\pi^{-1}$;
\item for each $\eps>0$ there exists $\eta>0$ and a compact set $\hat K\subset \hat I_{\T}\sm \bd \hat I$ such that any measure $\mu\in \M_f$ with $h(\mu)>\eps$ and $0<\lambda(\mu)<\infty$ has $\hat\mu(\hat K)>\eta$.
\end{enumerate}
\label{thm:cusp facts}
\end{teo}

\section{Projecting conformal measures}
\label{sec:appendix}

In Section~\ref{sec:ind} we outlined the relation between invariant measures for inducing schemes and for the original system.  Here we are concerned with proving analogous results for conformal measures.
We show how, given a suitable inducing scheme, the conformal measure for the induced system can be projected to the original system (and vice versa).  One way of doing this was described in \cite{Tod} where the structure of the Hofbauer extension was used.  Here we use a method inspired by that construction, but not directly appealing to it.  Note that one could also write the proofs here with notation from the Young tower.   The conformal measure so-constructed can by used to prove Theorem B via Proposition B.2 in \cite{IomTod11}.

\begin{lema}
A $\phi$-conformal measure $m_\phi$ for $(I,f)$ is also a $\Phi$-conformal measure for $(X,F)$ if $m_\phi(\cup_i X_i)=m_\phi(X)$.
\end{lema}

The lemma says that $m_\phi$ `lifts' to $(X,F)$. The proof is elementary since for $x\in X_i$,
$$dm_\phi(F(x))=dm_\phi(f^{\tau_i}(x))=e^{-S_{\tau_i}\phi(x)}dm_\phi(x)=e^{-\Phi(x)}dm_\phi(x),$$
and the complement of $\cup_i X_i$ has zero measure.
The point here is to prove the other direction: that  for certain inducing schemes $(X,F, \Phi)$, any $\Phi$-conformal measure, if it exists, projects to a $\phi$-conformal measure.  To do this we will work with a certain type of inducing schemes:

\textbf{Condition $(*)$:} An inducing scheme $(X,F, \tau)$ satisfies condition $(*)$ if for any $x\in (X,F)^\infty$, if $y, y'\in  (X,F)^\infty$ have $f^k(y)=f^{k'}(y')=x$ for $k, k'\in \N_0$, then there exists $n\in \N$ such that $k+n$ and $k'+n$ are inducing times for $y$ and $y'$ respectively.

It will turn out that this condition is satisfied by a natural class of inducing schemes, one way of obtaining which is the Hofbauer extension.  With this in mind, one can compare Condition $(*)$ to condition $(\dag)$ in Appendix~\ref{sec:Hof}.

\begin{teo}
Suppose that $f\in \F$ and $(X,F,\tau)$ is an inducing scheme satisfying condition $(*)$.  If $\phi:I\to[-\infty, \infty]$ has an induced version $\Phi:X\to [-\infty, \infty)$ with a $\Phi$-conformal probability measure $m_\Phi$, then $m_\Phi$ projects to a $\sigma$-finite $\phi$-conformal measure $m_\phi$, where $\mu_\phi\ll m_\phi$.

Moreover, $m_\phi$ is a finite measure if $\phi$ is bounded below: the measure is also finite if for each $x\in \{\phi=-\infty\}$ there exists $n\in \N_0$ such that $f^n(x)\in X$ and for any $y$ in neighbourhood of $x$,  $\phi(f^k(y))<\infty$ for $1\le k\le n-1$.
\label{thm:proj conf}
\end{teo}

Note that we allow $\Phi$ to be $-\infty$, but not $\infty$: indeed for each $i$, and $x\in X_i$ we allow $S_j\phi(x)=-\infty$ for any $1\le j\le \tau_i-1$, but not $S_j\phi(x)=+\infty$.

\begin{proof}
First note that $\Phi<\infty$ implies that for any set $A\subset X$, $m_\Phi(A)>0$ implies $m_\Phi(F(A))>0$.  This means that  no set of positive measure can leave $\cup_iX_i$ under iteration of $F$, i.e.,  for all $k\in \N$, $m_\Phi(F^{-k}(\cup_iX_i))=1$ and thus $m_\Phi(X\cap (X,F)^\infty)=1$.
We will spread the measure $m_\Phi$ onto $(X,F)^\infty$  as follows.

Suppose that $x\in (X,F)^\infty$  is contained in some set $f^k(X_i)$ for $0\le k\le \tau_i-1$.  There may be many such pairs $(i, k)$, but we pick one arbitrarily and then later show that we could have chosen any and obtained the same result.
There exists a unique $y\in X_i$ such that $f^k(y)=x$.  We define $\nu_\phi$ so that for any $j\in \N_0$,
 $$d\nu_\phi(f^j(x))=e^{-S_{k+j}(y)}dm_\Phi(y).$$
Clearly this gives a conformal measure locally since for $j\in \N$,
$$d\nu_\phi(f^j(x))=e^{-S_{k+j}\phi(y)}dm_\Phi(y)= e^{-S_j\phi(x)}e^{-S_k\phi(y)}dm_\Phi(y)=e^{-S_j\phi(x)}d\nu_\phi(x).$$
Note that if $k+j=\tau_i$ then $f^{k+j}(x)=F(y)$ and
$$d\nu_\phi(F(y))=d\nu_\phi(f^{k+j}(x))=e^{-S_{k+j}\phi(y)}dm_\Phi(y)=e^{-\Phi(y)}dm_\Phi(y)=dm_\Phi(F(y)),$$
by the $\Phi$-conformality of $m_\Phi$.  This also extends to the case when $k+j=\tau^p(y)$ for $p\in \N$, we obtain
\begin{equation}
d\nu_\phi(f^{j}(x))=d\nu_\phi(f^{k+j}(y))=e^{-S_{k+j}\phi(y)}dm_\Phi(y)=e^{-S_p\Phi(y)}dm_\Phi(y)=dm_\Phi(F^p(y)).
\label{eq:nu and Phi}
\end{equation}

To prove that the procedure given above is well-defined, we need to check that the same measure is assigned at $x$ when there is  $i'\neq i$ and $1\le k'\le \tau_{i'}-1$ such that $x$ is also contained in $f^{i'}(X_{i'})$ and $\tau_{i'}-k'=\tau_i-k$.  If we let $\nu_\phi'$ be the measure at $x$ obtained analogously to $\nu_\phi$ but with $i'$ and $k'$ in place of $i$ and $k$, and some point $y'\in X_{i'}$ in place of $y\in X_i$, we must show that $\nu_\phi'=\nu_\phi$.

By condition $(*)$ there exists $n$ such that $k+n$ and $k'+n$ are inducing times for $y$ and $y'$ respectively.  Therefore, as in \eqref{eq:nu and Phi},
\begin{align*}
d\nu_\phi'(x)&=e^{-S_n\phi(x)}d\nu_\phi'(f^n(x))=e^{-S_n\phi(x)}dm_\Phi(f^n(x))\\
&=e^{-S_n\phi(x)}d\nu_\phi(f^n(x))=d\nu_\phi(x),\end{align*}
so $\nu_\phi'=\nu_\phi$, as required.

For $\sigma$-finiteness of $\nu_\phi$, notice that for any $x\in \cup_{i\ge 1}\cup_{k=0}^{\tau_i-1}f^k(X_i)$, we can choose a single element $f^k(X_i)$ as above containing $x$ to apportion measure at $x$.  Since the resulting measure is finite, $\nu_\phi$ is $\sigma$-finite.

To prove that $\nu_\phi$ is actually finite when $\phi>-\infty$, we use a topological property of $f$: that for any open set $U\subset I$, there exists $n\in \N$ such that $I=f^n(U)$.  We take $U=X$, and note that $\nu_\phi(X)=m_\Phi(X)=1$. Then by the conformality of $\nu_\phi$ and the fact that $\phi$ is bounded below, we have
$$\nu_\phi(I)=\nu_\phi(f^n(X))\le \int_Xe^{-S_n\phi}~d\nu_\phi \le e^{-n\inf\phi}m_\Phi(X)<\infty.$$  The fact that $\nu_\phi\ll m_\phi$ follows from construction.

By this argument, the only possible obstacle to $\nu_\phi$ being finite in the general case is if there is a point $x\in I$ such that $\phi(x)=-\infty$: in this case, any neighbourhood of $x$ may have infinite measure.  We consider this possibility as in the final part of the statement of the theorem: we take $x\in \{\phi=-\infty\}$, and  $n\in \N_0$ such that $f^n(x)\in X$ and $\phi(f^k(x))<\infty$ for $1\le k\le n-1$.  Then there exists an interval $U\ni x$ such that $f^n(U)\subset X$ and $S_k\phi(y)<\infty$ for all $y\in U$.  It suffices to prove that $\nu_\phi(U)<\infty$.  This follows since
$$1=\nu_\phi(X)\ge \nu_\phi(U)=\int_Ue^{-S_n\phi}~d\nu_\phi\ge e^{-\sup_{y\in U} S_n\phi(y)}\nu_\phi(U).$$

We finish by setting $m_\phi=\nu_\phi/\nu_\phi(I)$.
\end{proof}

One of the main applications of this theorem is to maps $f\in\F$ with potential $-t\log|Df|$.  In particular we are interested in conformal measures for the potential $\psi_t:x\mapsto -t\log|Df(x)|-P(-t\log|Df|)$.  By Theorem~\ref{thm:proj conf}, in order to obtain a $\psi_t$-conformal measure it is sufficient
to find an inducing scheme $(X,F, \tau)$ satisfying condition $(*)$ with a $\Psi_t$-conformal measure $m_{\Psi_t}$ (where $\Psi_t=-t\log|DF|-\tau P(-t\log|DF|)$).  The existence of such schemes is studied in \cite{BruKel98}, \cite{PesSen08}, \cite{BruTod09} and \cite{IomTod10}.  We address this in the next section.

\subsection{Decent returns and natural conformal measures}

In this section we introduce a useful kind of inducing scheme, which in particular satisfies condition $(*)$.

Suppose that $A\subset I$ is an interval and $\delta>0$.  We define $(1+2\delta)A$ to be the interval concentric with $A$, but with length $(1+2\delta)|A|$ (we will always implicitly assume that $(1+2\delta)A$ is still contained in $I$).  For a set $B\subset I$, we say that $A$ is \emph{$\delta$-well inside $B$} if $(1+2\delta)A\subset B$.

The following is simply an application of the Koebe lemma, see \cite[Chapter IV]{MelStr93}, which applies in our case since $f\in \F$ has negative Schwarzian derivative.

\begin{lema}
Suppose that $f\in \F$ and $A\subset I$ is such that $f^n:(1+2\delta)A\to f^n((1+2\delta)A)$ is a diffeomorphism.  Then for all $x,y\in A$,
$$ \frac{\delta^2}{1+\delta^2}\le \frac{Df^n(x)}{Df^n(y)} \le \frac{1+\delta^2}{\delta^2}.$$
\label{lem:koebe}
\end{lema}

Given $f\in \F$, an interval $X\in I$ and $\delta>0$, we say that a point $x\in X$ makes a \emph{ $\delta$-decent return to $X$ at time $n\in \N$} if there exists an interval $A'(\delta, x)$ containing $x$ such that $f^n:A'(\delta, x)\to (1+2\delta)X$ is a diffeomorphism.  Let $A(\delta, x)\subset A'(\delta, x)$ be such that $f^n:A(\delta, x)\to X$ is a diffeomorphism.   The smallest such $n\in \N$ is called the  \emph{first $\delta$-decent return to $X$}.

The idea is that this will give us a very natural inducing scheme.  To ensure this, we need one more idea: we say that our interval $X$ is \emph{$\delta$-nice} for any $x, y\in X$, either  $A(\delta, x)= A(\delta, y)$ or  $A(\delta, x)$ and $A(\delta, y)$ intersect in at most one point.  One way to ensure this is to use the usual notion of `nice interval': an interval $X$ is \emph{nice} if $f^n(\bd X)\cap int(A)=\es$ for all $n\in \N$.  This in turn can, for example, be guaranteed if we let $p$ be a periodic point and let $X$ be an open interval bounded by points in the orbit of $p$, but not containing any points of that orbit.  By Lemma~\ref{lem:koebe} below, this process produces a `decent' inducing scheme.

It is proved in \cite{Bru95} that decent inducing schemes can be produced by the Hofbauer extension for $f$, and thus that such an inducing scheme $(X,F,\tau)$ satisfies condition $(*)$.

\begin{teo}
Given $f\in \F$, for each $t<t^+$ there exists an inducing scheme $(X,F,\tau)$ satisfying condition $(*)$ and a $\Psi_t$-conformal measure $m_{\Psi_t}$.
\label{thm:nat sch pot}
\end{teo}

This is essentially proved putting together the main theorem of  \cite{IomTod10} and the decent inducing schemes produced in  \cite{Bru95}. %We have $t_0\ge 1$ if there is no `wild attractor', for example if there is only one critical point and it is of quadratic order, or more simply if there is an acip.  If $f$ satisfies the Collet-Eckmann condition then $t_0>1$.

An immediate corollary of Theorems~\ref{thm:proj conf} and \ref{thm:nat sch pot} is:

\begin{coro}
Given $f\in \F$, for each $t<t^+$ there exists a $\psi_t$-conformal measure $m_{\psi_t}$.
\end{coro}

On can also apply this result to any H\"older potential $\phi:I\to \R$ satisfying $\sup \phi<P(\phi)$, see
\cite{BruTod08} and \cite{IomTod10}.  However, such conformal measures were already known, see  \cite{Kel85}.  Actually, it has been recently shown by Li and Rivera-Letelier \cite{lr} that for maps $f \in \F$ such that for every critical value $c$ we have $\lim_{n \to \infty} |Df^n(f(c))| = \infty$ every H\"older potential $\phi$ satisfies  $\frac{1}{n} \sup S_n \phi<P(\phi)$ and that this is enough to guarantee the existence and uniqueness of equilibrium states, conformal measures and real analyticity of the pressure function. This remarkable result allows for the improvement of Theorem B in \cite{IomTod11}.

\section{Equality of pointwise and Markov dimensions}  \label{sec:Markdim}

In \cite{OlsPet12}, work was done to explain a case of \cite[Proposition 3]{PolWe} which relates pointwise dimension and Markov dimension for a so-called EMR transformation $(I_0, T)$: for a precise definition of an EMR see \cite{PolWe}, but for our purposes here, just suppose that it satisfies all the properties of an inducing scheme described in Section~\ref{sec:ind}, but with $\tau\equiv 1$.  The key issue is the part where they show, under the assumption that the Markov dimension $\delta_\mu(x)$ exists at $x$, that $\underline d_\mu(x)  \ge \delta_\mu(x)$.
In this section we alter the definition of an EMR transformation in order to obtain this result.

For $I_0\subset I$, the EMR is a map $T:I_0'\to I_0$ where $I_0':=\cup_{i\in \N} I_{1,i}$ and for each $i\in \N$,  $T:I_{1,i}\to I_0$ is a diffeomorphism, and $\inf\{|DT(x)|:x\in I_0'\}>1$.  For $x\in I_0'$, if $T^n(x)$ is defined, let $I_n(x)$ denote the $n$-cylinder at $x$; recall that this is the maximal set around $x$ such that $T^n:I_n(x)\to I_0$ is a diffeomorphism.  We can list these intervals as $\{I_{n,1}, I_{n,2}, \ldots\}$. Let $I_0''$ be the set of points for which $T^n$ is defined for all $n\in \N_0$, i.e., $I_0''=\{x\in I_0':T^n(x)\in I_0'\text{ for all } n\ge1\}$.

Given that we have a Markov structure for our system,  and thus a natural partition into 1-cylinders, we can define the Markov dimension as usual for points in  $I_0''$.

By definition of an EMR, or an inducing scheme, there exists $K_1\ge1$ such that
\begin{equation}
\frac1{K_1}\le \frac{|I_n(x)|}{|DT^n(x)|^{-1}}\le K_1.\label{eq:cyl dis}
\end{equation}

Now recall that if $\mu$ is an equilibrium state (a Gibbs measure) for $\Phi$ then there exists $C\ge 1$ such that
\begin{equation}
\frac1C\le \frac{\mu(I_n(x))}{e^{-nP(\Phi)+\sum_{j=0}^{n-1}\Phi(T^j(y))}}\le C,
\label{eq:Gibbs}
\end{equation}
for all $y\in I_n(x)$.
Let $\overline\Phi(x)$ denote the Birkhoff average at $x$, if it exists.

The following is essentially proved in \cite{PolWe}, but with some minor errors, so we give the proof again here.

\begin{lema}
Suppose that for $x\in I_0''$, $\delta_\mu(x)$ and $\overline\Phi(x)$ exist.
Then $\overline d_\mu(x)\le \delta_\mu(x)$.
\label{lem:PW1}
\end{lema}

\begin{proof}
Given $r>0$ there exists a unique $n=n(r)$ such that $K_1|DT^n(x)|^{-1}<r\le K_1|DT^{n-1}(x)|^{-1}$.  So $r>|I_n(x)|$ by \eqref{eq:cyl dis} and $I_n(x) \subset B(x, |I_n(x)|)\subset B(x,r)$.
Then
$$\frac{\log\mu(B(x, r))}{\log r}\le \frac{\log(\mu(I_n(x)))}{\log(K_1|DT^{n-1}(x)|^{-1})}\le \frac{\log\mu(I_{n-1}(x))}{\log(K_1^2|I_{n-1}(x)|)} \frac{\log\mu(I_n(x))}{\log\mu(I_{n-1}(x))}.$$
By the Gibbs property \eqref{eq:Gibbs},
$$ \frac{\log\mu(I_n(x))}{\log\mu(I_{n-1}(x))}\asymp  \frac{-nP(\Phi)+\sum_{j=0}^{n-1}\Phi(T^j(x))}
{-(n-1)P(\Phi)+\sum_{j=0}^{n-2}\Phi(T^j(x))}.$$
If $\overline\Phi(x)$ exists, then $\lim_{n\to\infty}  \frac{\log\mu(I_n(x))}{\log\mu(I_{n-1}(x))}=1$, so taking limits in $r$ we prove the lemma.
\end{proof}

In order to prove a reverse inequality to that in the previous lemma, we change the definition of an EMR.  Observing that this new condition can be satisfied by the inducing schemes used in, for example, \cite{IomTod11}.

We extend the system as follows.  In the interval $I$ (which is the ambient space containing $I_0$), let $I_0^\ell,I_0^r$ be intervals adjacent to $I_0$ and set $\hat I_0:=I_0^\ell\cup I_0\cup I_0^r$.  We will assume that for any $n\in \N$ and $i\in \N$, the map $T^n:I_{n,i}\to I_0$ has an extension $\hat T^n$ so that for an interval $\hat I_0\supset \hat I_{n,i}\supset I_{n,i}$, the map $\hat T^n|_{\hat I_{n,i}}:\hat I_{n,i}\to \hat I_0$ is a diffeomorphism.  By `extension', we mean that $\hat T^n|_{I_{n,i}}= T^n|_{I_{n,i}}$.  %We note that there is room for confusion here since it may be the case that \textbf{GI: what do you mean here?} $\hat I_n^i\cap \hat I_n^i\neq \es$, so $\hat T^n $ isn't fantastically well defined.

We further assume that for there exists $K_2\ge1$ such that
\begin{equation}
\frac1{K_2}\le \frac{|I_n^\ell(x)|}{|I_n(x)|}, \frac{|I_n^r(x)|}{|I_n(x)|}\le K_2.\label{eq:hat cyl dis}
\end{equation}
  Also assume that the Gibbs property extends to $\hat T$,  that is \eqref{eq:Gibbs} holds as well as the following: there exists $K_3\ge 1$ such that
$$\frac1{K_3}\le \frac{\mu(\hat I_n(x))}{\mu(I_n(x))} \le K_3.$$
Note that in the setting of \cite{IomTod11}, $\frac{\mu(\hat I_n(x))}{\mu(I_n(x))}$ is simply $\frac{m_\phi(\hat I_0)}{m_\phi(I_0)}$ multiplied by some distortion constant (here $m_\phi$ is $\phi$-conformal measure on $I$).

\begin{lema}
Suppose that for $x\in I_0''$, $\delta_\mu(x)$ and $\overline\Phi(x)$ exist.
Then $\underline d_\mu(x)\ge \delta_\mu(x)$.
\label{lem:PW2}
\end{lema}

Therefore, combining Lemmas~\ref{lem:PW1} and \ref{lem:PW2}, if $\delta_\mu(x)$ and $\overline\Phi(x)$ exist,
then $\overline d_\mu(x)= \delta_\mu(x)$.

\begin{proof}
Given $r>0$ there exists a unique $n=n(r)$ such that $\frac{1}{K_1K_2}|DT^n(x)|^{-1}<r\le \frac{1}{K_1K_2}|DT^{n-1}(x)|^{-1}$.  So $r\le \frac{|I_{n-1}(x)|}{K_2}\le |I_{n-1}^\ell(x)|, |I_{n-1}^r(x)|$ by \eqref{eq:cyl dis} and \eqref{eq:hat cyl dis}, which implies that $B(x, r)\subset \hat I_{n-1}(x)$.
Then
$$\frac{\log\mu(B(x, r))}{\log r}\ge \frac{\log(\mu(\hat I_{n-1}(x)))}{\log\left(\frac{1}{K_1K_2}|DT^{n}(x)|^{-1}\right)}\ge \frac{\log\left(K_3\mu(I_n(x))\right)}{\log\left(\frac{1}{K_1^2K_2}|I_n(x)|\right)} \frac{\log\left(K_3\mu(I_{n-1}(x))\right)}{\log\left(K_3\mu(I_n(x))\right)}.$$
As in Lemma~\ref{lem:PW1}, taking limits in $r$ proves the lemma.
\end{proof}

The rest of \cite[Proposition 3]{PolWe} follows similarly.


\begin{thebibliography}{99}

%\bibitem[Aa]{Aaro_book} J.\ Aaronson,  \emph{An introduction to infinite ergodic theory,} Mathematical Surveys and Monographs, 50. American Mathematical Society, Providence, RI, 1997.

\bibitem[A]{Abr59}   L.M.\ Abramov, \emph{On the entropy of a flow,} Dokl. Akad. Nauk SSSR \textbf{128} (1959) 873--875.

\bibitem[AL]{AviLyu08} A.\ Avila, M.\ Lyubich,
\emph{ Hausdorff dimension and conformal measures of Feigenbaum Julia sets,}
 J. Amer. Math. Soc. \textbf{21} (2008) 305--363.




\bibitem[BaS]{BarSc} L.\ Barreira, J.\ Schmeling,
  \emph{Sets of ``non-typical'' points have full topological entropy and full Hausdorff dimension}, Israel J. Math. \textbf{116} (2000) 29--70.


\bibitem[BI]{bi} L. Barreira, G. Iommi \emph{Multifractal analysis and phase transitions for hyperbolic and parabolic horseshoes.}  Israel Journal of Mathematics \textbf{181} (2011) 347--379.


%\bibitem[Ba]{ba} L.\ Barreira,  \emph{Dimension and recurrence in hyperbolic dynamics.} Progress in Mathematics, 272. Birkh\"auser Verlag, Basel, 2008.

%\bibitem[Bo1]{bo1}
%R.\ Bowen, \emph{Symbolic dynamics for hyperbolic flows}, Amer. J.
%Math. \textbf{95} (1973) 429--460.

%\bibitem[Bo2]{bo2} R.\ Bowen, \emph{Some systems with unique equilibrium states,}
% Math. Systems Theory \textbf{8} (1974) 193--202.


\bibitem[BC]{BenCar} M.\ Benedicks, L.\ Carleson,
{\em On iterations of $1 - ax^2$ on $(-1, 1)$,}
Ann. of Math. {\bf 122} (1985) 1--25.

\bibitem[Bo]{bo}  R.\ Bowen, \emph{Equilibrium states and the ergodic theory of Anosov diffeomorphisms.} Second revised edition. With a preface by David Ruelle. Edited by Jean-Ren\'e Chazottes. Lecture Notes in Mathematics, 470. Springer-Verlag, Berlin, 2008.

\bibitem[BB]{BruBru04} K.M.\ Brucks, H.\ Bruin,
\emph{Topics from one-dimensional dynamics,}
London Mathematical Society Student Texts, 62. Cambridge University Press, Cambridge, 2004.

\bibitem[B]{Bru95}
H.\ Bruin,
\emph{Induced maps, {M}arkov extensions and invariant measures in
  one-dimensional dynamics,}
  Comm.\ Math.\ Phys.\ \textbf{168} (1995) 571--580.

%\bibitem[Br]{Brminim} H.\ Bruin,
%{\em Minimal Cantor systems and unimodal maps,}
%J. Difference Equ. Appl. {\bf 9} (2003) 305--318.

\bibitem[BK]{BruKel98}
H.\ Bruin, G.\ Keller.
\emph{Equilibrium states for {$S$}-unimodal maps,}
Ergodic Theory Dynam.\ Systems \textbf{18} (1998) 765--789.


\bibitem[BRSS]{BRSS} H.\ Bruin, J.\ Rivera-Letelier, W.\ Shen, S.\ van Strien,
{\em Large derivatives, backward contraction and invariant densities
for interval maps,} Invent. Math. {\bf 172} (2008) 509--593.


\bibitem[BT1]{BruTod08} H.\ Bruin, M.\ Todd, {\em Equilibrium states
    for potentials with $\sup \phi - \inf \phi < h_{top}(f)$,} Comm.
    Math. Phys. {\bf 283} (2008) 579--611.

\bibitem[BT2]{BruTod09} H.\ Bruin, M.\ Todd,
\emph{Equilibrium states for interval maps: the potential $ -t \log|Df|$,}
 Ann. Sci. \'Ec. Norm. Sup\'er. {\bf 42} (2009) 559--600.



\bibitem[BT3]{BruTod12} H.\ Bruin, M.\ Todd,
\emph{Transience and thermodynamic formalism for infinitely branched interval maps,}
J.\ London Math.\ Soc. \textbf{ 86} (2012) 171--194

\bibitem[BT4]{BruTod12a} H.\ Bruin, M.\ Todd,
\emph{Wild attractors and thermodynamic formalism,}
Preprint (arXiv1202.1770).

\bibitem[BuS]{BuSar} J.\ Buzzi, O.\ Sarig,
\emph{Uniqueness of equilibrium measures for countable Markov shifts and multidimensional piecewise expanding maps,}
Ergodic Theory Dynam. Systems {\bf 23} (2003) 1383--1400.

%\bibitem[ChH]{chh} J.R.\ Chazottes, M.\ Hochman,
 %\emph{On the zero-temperature limit of Gibbs states,}
% Comm. Math. Phys. \textbf{297} (2010) 265--281.

\bibitem[CLP]{ColLiePor87} P.\ Collet, J.L.\ Lebowitz, A.\ Porzio,
The dimension spectrum of some dynamical systems. In \emph{
Proceedings of the symposium on statistical mechanics of phase transitionsÑmathematical and physical aspects (Trebon, 1986),} volume 47, pages 609--644, 1987.


\bibitem[CR]{CorRiv12} D.\ Coronel, J.\ Rivera-Letelier,
\emph{ Low-temperature phase transitions in the quadratic family,}
Preprint (arXiv:1205.1833).

%\bibitem[CRL]{CJ} M.I.\ Cortez, J.\ Rivera-Letelier,
%\emph{Invariant measures of minimal post-critical sets of logistic maps,}
%Israel J. Math. 176 (2010) 157--193.

%\bibitem[C1]{Cyr_CMS} V.\ Cyr,  \emph{Countable Markov shifts with Transient Potentials,} `Advance Access' Proc LMS, 2011.

%\bibitem[C2]{Cyr_thes} V.\ Cyr, \emph{Transient Markov Shifts}
% Ph.D. thesis, The Pennsylvania State University, 2010.

\bibitem[CS]{CyrSar} V.\ Cyr, O.\ Sarig, \emph{Spectral Gap and Transience for Ruelle Operators on Countable Markov Shifts,}
    Comm. Math. Phys. \textbf{292} (2009) 637--666.

\bibitem[D1]{Dob09} N.\ Dobbs,
{\em Renormalisation induced phase transitions for unimodal maps,}
Comm. Math. Phys. {\bf 286} (2009) 377--387.

\bibitem[D2]{Dob08} N.\ Dobbs,
\emph{On cusps and flat tops,}
Preprint (arXiv:0801.3815).


\bibitem[Dobr]{Dobr} R.L.\ Dobru\v sin,
\emph{Description of a random field by means of conditional probabilities and conditions for its regularity,}
Teor. Verojatnost. i Primenen \textbf{13} (1968) 201--229.

\bibitem[EP]{EckPro86} J.-P.\ Eckmann, I.\ Procaccia,
\emph{Fluctuations of dynamical scaling indices in nonlinear systems,}
 Phys. Rev. A, \textbf{34} (1986) 659--661.

\bibitem[F]{Fis72} M.E.\ Fisher,
\emph{On discontinuity of the pressure,}
Comm. Math. Phys. \textbf{ 26} (1972)  6--14.



\bibitem[GPR]{GelPrzRam10} K.\ Gelfert, F.\ Przytycki, M.\ Rams,
 \emph{Lyapunov spectrum for rational maps,}  Math. Annalen.
 \textbf{348} (2010) 965--1004


\bibitem[GPRR]{gprr} K.\ Gelfert, F.\ Przytycki, M.\ Rams, J.\ Rivera-Letelier \emph{Lyapunov spectrum for exceptional rational maps.} Preprint (arXiv:1012.2593).

\bibitem[GR]{GelRam09} K.\ Gelfert, M.\ Rams,  \emph{The Lyapunov spectrum of some parabolic systems,} Ergodic Theory Dynam. Systems {\bf 29} (2009) 919-940.


\bibitem[GW]{gw} E.\ Glasner, B.\ Weiss, \emph{Kazhdan's property T and the geometry of the collection of invariant measures,} Geom. Funct. Anal. \textbf{7} (1997) 917--935.

\bibitem[GSS1]{GraSanSwi04} J.\ Graczyk, D.\ Sands, G.\ \'Swi\c atek,
\emph{Metric attractors for smooth unimodal maps,}
 Ann.\ of Math.\ (2) \textbf{159} (2004) 725--740.
%
%\bibitem[GSS2]{GraSanSwi05} J.\ Graczyk, D.\ Sands, G.\ \'Swi\c atek,
%\emph{Decay of geometry for unimodal maps: negative Schwarzian case.} Ann. of Math. (2) 161 (2005), no. 2, 613--677.


%\bibitem[Fe]{fe}  W. \ Feller, \emph{An introduction to probability theory and its applications.} Vol. I. Third edition John Wiley and Sons, Inc., New York-London-Sydney 1968.

\bibitem[Gu1]{Gur69} B.M.\ Gurevi\v c,
\emph{Topological entropy for denumerable Markov chains,}
Dokl. Akad. Nauk SSSR {\bf 10} (1969) 911--915.

\bibitem[Gu2]{Gur70} B.M.\ Gurevi\v c,
\emph{Shift entropy and Markov measures in the path space of a denumerable graph,}
Dokl. Akad. Nauk SSSR {\bf 11} (1970) 744--747.

%\bibitem[Ha]{Halmos} P.R.\ Halmos, \emph{Lectures on ergodic theory,} Chelsea Publishing Co., New York 1960.

\bibitem[HJK]{HalJenKad86} T.C.\ Halsey, M.H.\ Jensen, L.P.\ Kadanoff, I.\ Procaccia, B.I.\ Shraiman,
\emph{ Fractal measures and their singularities: the characterization of strange sets,}
 Phys. Rev. A (3) \textbf{33} (1986) 1141--1151.


\bibitem[HMU]{hmu} P. \ Hanus,  D. \ Mauldin, M. \ Urba\'nski, \emph{Thermodynamic formalism and multifractal analysis of conformal infinite iterated function systems}, Acta Math. Hungarica \textbf{96} (2002) 27--98.

\bibitem[H1]{Hof77} F.\ Hofbauer,
{\em Examples for the nonuniqueness of the equilibrium state,}
Trans. Amer. Math. Soc.  {\bf 228}  (1977) 223--241.


\bibitem[H2]{Hofdim} F.\ Hofbauer,
{\em Local dimension for piecewise monotonic maps on the interval,}
Ergodic Theory Dynam. Systems {\bf 15} (1995) 1119--1142.

\bibitem[HK]{HofKel82} F.\ Hofbauer, G.\ Keller, \emph{Equilibrium states for piecewise monotonic transformations,} Ergodic Theory Dynam. Systems \textbf{2} (1982) 23--43.



\bibitem[I]{Iom05} G. Iommi, \emph{Multifractal analysis for countable Markov shifts,} Ergodic Theory Dynam. Systems \textbf{25} (2005) 1881--1907.
%
%\bibitem[I2]{Iom10} G.\ Iommi, \emph{Multifractal analysis of Lyapunov exponent for the backward continued fraction map,}   Ergodic Theory Dynam. Systems. \textbf{30} (2010) 211--232.
%
%\bibitem[IK]{IomKiw09} G.\ Iommi, J.\ Kiwi, \emph{The Lyapunov spectrum is not always concave,}  J. Stat. Phys. \textbf{135} 535-546 (2009).


\bibitem[IT1]{IomTod10} G.\ Iommi, M.\ Todd, \emph{Natural equilibrium states  for multimodal maps,} Comm. Math. Phys. \textbf{300} (2010) 65--94.

\bibitem[IT2]{IomTod11} G.\ Iommi, M.\ Todd, \emph{Dimension theory for multimodal maps,}
 Ann. Henri Poincar\'e \textbf{12} (2011) 591--620.

\bibitem[J]{Jak} M.V.\ Jakobson,
{\em Absolutely continuous invariant measures for one-parameter families of one-dimensional maps,}
Comm. Math. Phys. {\bf 81} (1981) 39--88.

\bibitem[K1]{Kel85}
G.\ Keller,
\emph{Generalized bounded variation and applications to piecewise monotonic  transformations,}
Z.\ Wahrsch.\ Verw.\ Gebiete \textbf{ 69} (1985) 461--478.

\bibitem[K2]{Kel89}
G.\ Keller,
\emph{Lifting measures to {M}arkov extensions,}
 Monatsh.\ Math.\ \textbf{108} (1989) 183--200.

\bibitem[K3]{Kel98} G.\ Keller, {\em Equilibrium states in ergodic theory,} London Mathematical Society Student Texts, 42. Cambridge University Press, Cambridge, 1998.

\bibitem[KMS]{kms} M. \ Kesseb\"ohmer, S. \ Munday,  B. \ Stratmann, \emph{Strong renewal theorems and Lyapunov spectra for alpha-Farey and alpha-L\"uroth systems}
Ergodic Theory Dynam. Systems \textbf{32} (2012) 989--1017.

\bibitem[Le]{Led81} F.\ Ledrappier,
\emph{ Some properties of absolutely continuous invariant measures on an interval,}
Ergodic Theory Dynam. Systems {\bf 1} (1981) 77--93.

\bibitem[LOR]{LepOliRio11} R.\ Leplaideur, K.\ Oliveira, I.\ Rios,
\emph{ Equilibrium states for partially hyperbolic horseshoes,}
 Ergodic Theory Dynam. Systems \textbf{ 31} (2011) 179--195.

\bibitem[LR]{LepRio09} R.\ Leplaideur, I. Rios,
\emph{ On $t$-conformal measures and Hausdorff dimension for a family of non-uniformly hyperbolic horseshoes,}
Ergodic Theory Dynam. Systems \textbf{ 29} (2009) 1917--1950.

\bibitem[LiR]{lr} H.\ Li, J.\ Rivera-Letelier  \emph{Equilibrium states of interval maps for hyperbolic potentials.} Preprint, arXiv:1210.6952

\bibitem[LOS]{los} J.\ Lindenstrauss, G.\ Olsen, Y.\ Sternfeld,
 \emph{The Poulsen simplex,} Ann. Inst. Fourier (Grenoble) \textbf{28} (1978) 91--114.

%\bibitem[LSV]{LivSaVai} C.\ Liverani, B.\ Saussol, S.\ Vaienti,
 %\emph{Probabilistic approach to intermittency,}
 % Ergodic Theory Dynam. Systems \textbf{19} (1999) 671--685.

\bibitem[Ly]{Lyu94} M.\ Lyubich,
\emph{ Combinatorics, geometry and attractors of quasi-quadratic maps,}
Ann.\ of Math.\ (2) \textbf{ 140} (1994) 347--404.


\bibitem[MS]{MakSmi03} N.\ Makarov, S.\ Smirnov,
\emph{On thermodynamics of rational maps. II. Non-recurrent maps,}
J. London Math. Soc. (2) \textbf{67} (2003) 417--432.

%\bibitem[MP]{ManPom}   P.\ Manneville, Y.\ Pomeau,
%\emph{Intermittent transition to turbulence in dissipative dynamical systems,}
%Comm. Math. Phys. {\bf 74} (1980) 189--197.

%\bibitem[MP]{MarkPaul} N.G.\ Markley, M.E.\ Paul,  \emph{Equilibrium states of grid functions,} Trans.  Amer. Math. Soc. \textbf{274} (1982) 169--191.

\bibitem[MU1]{MUifs} R.\ Mauldin, M.\ Urba\'nski,
\emph{Dimensions and measures in infinite iterated function systems,} Proc. London Math. Soc. (3) \textbf{73} (1996) 105--154.

\bibitem[MU2]{muGIBBS} R.\ Mauldin, M.\ Urba\'nski,
\emph{Gibbs states on the symbolic space over an infinite alphabet,} Israel J. Math. \textbf{125} (2001) 93--130.

\bibitem[MeSt]{MelStr93}  W.\ de Melo, S.\ van Strien, \emph{One-dimensional dynamics.} Ergebnisse der Mathematik und ihrer Grenzgebiete (3) [Results in Mathematics and Related Areas (3)], 25. Springer-Verlag, Berlin, 1993.

\bibitem[N]{Nak00} K. Nakaishi,
\emph{Multifractal formalism for some parabolic maps,}
 Ergodic Theory Dynam. Systems \textbf{ 20} (2000) 843--857.

\bibitem[NS]{NowSan98} T.\ Nowicki, D.\ Sands,
{\em Non-uniform hyperbolicity and universal bounds for $S$-unimodal maps,}
Invent. Math. {\bf 132}  (1998) 633--680.

%\bibitem[O]{Olivier} E.\ Olivier, \emph{Multifractal analysis in symbolic dynamics and distribution of pointwise dimension for $g$-measures,}
    % Nonlinearity \textbf{12} (1999) 1571--1585.

\bibitem[OP]{OlsPet12} L.\ Olsen, I.\ Petrykiewicz,
\emph{Geometric and Symbolic Local Dimensions of Continued Fractions,}
Manuscript, University of St Andrews.

\bibitem[PT]{PacTod10}  M.J.\ Pacifico, M.\ Todd,
\emph{Thermodynamic formalism for contracting Lorenz flows,}
 J.\ Stat.\ Phys.\ \textbf{139} (2010) 159--176.

%\bibitem[PP]{ParPol}
%W.\ Parry, M.\ Pollicott, \emph{Zeta Functions and the Periodic
%Orbit Structure of Hyperbolic Dynamics}, Ast\'erisque 187--188,
%1990.




\bibitem[P]{Pesbook} Y.\ Pesin,  \emph{Dimension Theory in Dynamical Systems,} CUP 1997.


\bibitem[PS]{PesSen08} Y.\ Pesin, S.\ Senti,
\emph{Equilibrium measures for maps with inducing schemes,}
J. Mod. Dyn. {\bf 2} (2008) 1--31.

%\bibitem[PZ]{PesZhang}  Y.\ Pesin, K.\ Zhang, \emph{Phase transitions for uniformly expanding maps,} J. Stat. Phys. \textbf{122} (2006) 1095--1110.

%\bibitem[PP]{ParPol} W. \ Parry, M. \ Pollicott, \emph{Zeta functions and the periodic orbit structure of hyperbolic dynamics.} AstŽrisque No. 187-188 (1990), 268 pp.



\bibitem[Poi]{poi} H.\ Poincar\'e,  \emph{Sur le probl\'eme des trois corps et les \'equations de la dynamique,} Acta Math.  {\bf 13} Number 1 (1890).


\bibitem[PolW]{PolWe} M.\ Pollicott, H.\ Weiss, \emph{Multifractal analysis of Lyapunov exponent for continued fraction and Manneville-Pomeau transformations and applications to Diophantine approximation,}  Comm. Math. Phys. \textbf{ 207} (1999) 145--171.


\bibitem[Pr]{Prz93} F.\ Przytycki,
{\em Lyapunov characteristic exponents are nonnegative,}
Proc. Amer. Math. Soc. {\bf 119} (1993) 309--317.


\bibitem[PR]{RivPrz11}   F.\ Przytycki, J.\ Rivera-Letelier
\emph{Nice inducing schemes and the thermodynamics of rational maps,}
Comm.\ Math.\ Phys.\ \textbf{301} (2011) 661--707.


\bibitem[PrU]{PrzUrb_book} F.\ Przytycki, M. Urba\'nski,
\emph{Fractals in the Plane, Ergodic Theory Methods,}
Cambridge University Press 2010.


\bibitem[Ra]{ra} D.A.\ Rand,  \emph{The singularity spectrum f(a)  for cookie-cutters.}
Ergodic Theory Dynam. Systems \textbf{9} (1989) 527--541.

%\bibitem[Ra]{Ratner} M.\ Ratner, \emph{Markov partitions for Anosov flows on $n$-dimensional manifolds}, Israel J. Math. \textbf{15} (1973) 92--114.

\bibitem[R]{Riv12} J.\ Rivera-Letelier,
\emph{ Asymptotic expansion of smooth interval maps}
Preprint (arXiv:1204.3071).

\bibitem[RS]{RivShe10}  J.\ Rivera-Letelier, W.\ Shen,
{\em Statistical properties of one-dimensional maps under weak hyperbolicity assumptions,} Preprint, arXiv:1004.0230.

\bibitem[Ro]{Rov93} A.\ Rovella,
\emph{The dynamics of perturbations of the contracting Lorenz attractor,}
Bol. Soc. Brasil. Mat. (N.S.) \textbf{24} (1993) 233--259.

%\bibitem[Ru1]{ru1} D.\ Ruelle, \emph{Statistical mechanics on a compact set with $Z^{v}$ action satisfying expansiveness and specification,} Trans. Amer. Math. Soc. 187 (1973) 237--251.

\bibitem[Ru]{Ru_book}  D.\ Ruelle, \emph{Thermodynamic formalism. The mathematical structures of equilibrium statistical mechanics,} Second edition. Cambridge Mathematical Library. Cambridge University Press, Cambridge, 2004.

\bibitem[S1]{Sar99} O.\ Sarig,
{\em Thermodynamic formalism for countable Markov shifts,}
Ergodic Theory Dynam. Systems {\bf 19} (1999) 1565--1593.

\bibitem[S2]{Sar01} O.\ Sarig,
 \emph{Thermodynamic formalism for null recurrent potentials,} Israel J. Math. \textbf{121} (2001) 285--311.

% \bibitem[S3]{Saran} O.\ Sarig,
%\emph{On an example with topological pressure which is not analytic,} C.R. Acad. Sci. Serie I: Math. {\bf 330} (2000) 311-315.

\bibitem[S3]{Sar01a} O.\ Sarig,
\emph{Phase transitions for countable Markov shifts,}
Comm. Math. Phys. {\bf 217} (2001) 555--577.

\bibitem[S4]{SarBIP} O.\ Sarig,
\emph{Existence of Gibbs measures for countable Markov shifts,}
Proc. Amer. Math. Soc. {\bf 131} (2003) 1751--1758.

%\bibitem[S5]{Sarcrit} O.\ Sarig,
%\emph{Critical exponents for dynamical systems,}
%Comm. Math. Phys. {\bf 267} (2006) 631-667.

\bibitem[S4]{Sar09}  O.\ Sarig, \emph{Lectures Notes on Thermodynamic Formalism for Topological Markov Shifts}
(2009) (\url{http://www.wisdom.weizmann.ac.il/~sarigo/}).

%\bibitem[S5]{snunif}   O.\ Sarig, \emph{Symbolic dynamics for surface diffeomorphisms with positive topological entropy} Preprint  (arXiv:1105.1650).

\bibitem[ST1]{SenTak11} S.\ Senti, H. Takahasi,
\emph{Equilibrium measures for the H\'enon map at the first bifurcation,}
Preprint  (arXiv:1110.0601).

\bibitem[ST2]{SenTak12} S.\ Senti, H. Takahasi,
\emph{Equilibrium measures for the H\'enon map at the first bifurcation: uniqueness and geometric/statistical properties,}
Preprint  (arXiv:1209.2224).

%\bibitem[Sm]{smale}  S.\  Smale, \emph{Differentiable dynamical systems,}  Bull. Amer. Math. Soc. \textbf{73} (1967) 747--817.


\bibitem[Si]{Sinai} J.G.\ Sinai,
\emph{Gibbs measures in ergodic theory,}
Uspehi Mat. Nauk \textbf{27} (1972) 21--64.


\bibitem[T]{Tod} M.\ Todd,  \emph{Multifractal analysis for multimodal maps,} Preprint.

\bibitem[V1]{Ver62} D.\ Vere-Jones, \emph{Geometric ergodicity in denumerable Markov chains,} Quart. J. Math. Oxford Ser. (2) \textbf{13} (1962) 7--28.

\bibitem[V2]{Ver67}   D.\ Vere-Jones, \emph{Ergodic properties of nonnegative matrices I,} Pacific J. Math. \textbf{22} (1967) 361--386.

%\bibitem[W1]{waTA}  P.\ Walters, \emph{Ruelle's operator theorem and $g$-measures,} Trans. Amer. Math. Soc. 214 (1975) 375--387.

%\bibitem[W2]{wbeta}  P.\ Walters, \emph{Equilibrium states for $\beta$-transformations and related transformations,} Math. Z. \textbf{159} (1978) 65--88.

%\bibitem[W3]{wa1} P.\ Walters, \emph{Invariant measures and equilibrium states for some   mappings which expand distances, }
%Trans. Amer. Math. Soc. \textbf{236} (1978) 121--153.

\bibitem[W]{Waltbook} P.\ Walters, \emph{An introduction to ergodic theory}, Graduate Texts in Mathematics 79. Springer-Verlag, New York-Berlin, 1982.

%\bibitem[W4]{wa2} P.\ Walters, \emph{Convergence of the Ruelle operator for a function satisfying Bowen's condition,} Trans. Amer. Math. Soc. \textbf{353} (2001) 327--347.

%\bibitem[W5]{wa3} P.\ Walters, \emph{A natural space of functions for the Ruelle operator theorem,} Ergodic Theory Dynam. Systems \textbf{27} (2007) 1323--1348.


\bibitem[We]{we} H.\ Weiss, \emph{The Lyapunov spectrum for conformal expanding maps and axiom-A surface diffeomorphisms,} J. Statist. Phys. \textbf{95} (1999) 615--632.

\bibitem[Y]{You99} L.S.\ Young,
{\em  Recurrence times and rates of mixing,}
Israel J. Math. {\bf 110}  (1999) 153--188.

\bibitem[Z]{Zwe05} R.\ Zweimuller, \emph{Invariant measures for general(ized) induced transformations,} Proc. Amer. Math. Soc. \textbf{133} (2005) 2283--2295.

\end{thebibliography}
\end{document}